\documentclass[11pt]{article}
\usepackage{amsmath,amsthm,amsfonts,amssymb,amscd, amsxtra, mathrsfs}
\usepackage{url}
\usepackage[margin=2.7 cm,nohead]{geometry}
\usepackage{color}

\usepackage{xcolor}
\usepackage{float}
\usepackage{threeparttablex,color,soul,colortbl,hhline,rotating,booktabs,longtable,multirow,makecell}
\definecolor{lightgray}{gray}{0.95}

\newtheorem{theorem}{Theorem}
\newtheorem{lemma}[theorem]{Lemma}
\newtheorem{definition}{Definition}
\newtheorem{corollary}[theorem]{Corollary}
\newtheorem{proposition}[theorem]{Proposition}

\newtheorem{remark}{Remark}

\usepackage{algorithm}
\usepackage{algpseudocode}

\newcommand{\ds}{\displaystyle}
\newcommand{\R}{\mathbb{R}}
\newcommand{\tr}{\textrm{tr}}
\newcommand{\rank}{\textrm{rank}}

\algnewcommand{\Input}[1]{%
  \State \textbf{Input:} {\raggedright #1}
  
}

\algnewcommand{\Initialize}[1]{%
  \State \textbf{Initialize:}
  \Statex \hspace*{\algorithmicindent}\parbox[t]{.8\linewidth}{\raggedright #1}
}

\algnewcommand{\Output}[1]{%
  \State \textbf{Output:} {\raggedright #1}
}

\begin{document}
\title{Inexact gradient projection method with relative error  tolerance}
\author{
A. A. Aguiar \thanks{Instituto de Matem\'atica e Estat\'istica, Universidade Federal de Goi\'as,  CEP 74001-970 - Goi\^ania, GO, Brazil, E-mails: {\tt  ademiraguia@discente.ufg.br}, {\tt  orizon@ufg.br}, {\tt  lfprudente@ufg.br}. The authors was supported in part by  CNPq grants 305158/2014-7 and 302473/2017-3,  FAPEG/PRONEM- 201710267000532 and CAPES.}
\and
O.  P. Ferreira  \footnotemark[1]
\and
L. F. Prudente \footnotemark[1]
}

\maketitle

\maketitle
\begin{abstract}
A gradient projection method with feasible inexact projections is proposed in the present paper. The inexact projection is performed using a relative error tolerance. Asymptotic convergence analysis and iteration-complexity bounds of the method employing constant and Armijo step sizes are presented. Numerical results are reported illustrating  the potential advantages of considering inexact projections instead of exact ones  in  some medium scale instances of a least squares problem over the spectrohedron.
\end{abstract}



\section{Introduction}
In this paper, we address general  constrained convex optimization problems of the form
\begin{equation} \label{eq:OptP}
	\min \{ f(x) :~ x\in C\},
\end{equation}
where $C$ is a closed and convex subset of $\mathbb{R}^n$ and $f:\mathbb{R}^n \to {\mathbb R}$ is a continuously differentiable function. Denote by $f^*:= \inf_{x\in C} f(x)$ the optimal value  of \eqref{eq:OptP} and by  $\Omega^*$  its  solution set, which we will assume to be non-empty unless the contrary is explicitly stated. The Problem~\eqref{eq:OptP} is a basic optimization issue, it has appeared  in several areas of science and engineering, including   machine learning, control theory and signal processing, see for example \cite{Bottou_Curtis_Nocedal2018, Boyd_Ghaoui_Ferron1994, Figueiredo2007, Ma_Hu_Gao2015, Sra_Nowozin_Wright2012}. In the present paper, we are  interested in  gradient-type  algorithms   to solve it. 

The gradient projection method (GPM) is the one of the most oldest method to solve Problem \eqref{eq:OptP}, its convergence properties go back to the works of Goldstein \cite{Goldstein1964} and Levitin and Polyak \cite{Polyak_Levitin1966}. After these works, many variants of it have appeared throughout the years, resulting in a wide literature on the subject, see, for example, \cite{yunier_roman2010, Bertsekas1976, Bertsekas1999, Fan_Wang_Yan2019, Figueiredo2007, Gong2011,  Max_Jefferson_Renato2020, Iusem2003, Patrascu_Necoara2018, Zhang_Wang_Yang2019}. The GPM has attracted the attention of the scientific community working in optimization, mainly due to its  simplicity and easy implementation. Besides, since this method uses only first order derivatives, it is often considered as a scalable solver for large-scale optimization problems, see \cite{Guanghui2013, More1989, More1990, Nesterov_Nemirovski2013, Sra_Nowozin_Wright2012, tang_golbabaee_davies2017}. At each iteration, the classical GPM moves along the direction of the negative gradient, and then projects the iterate onto $C$ if it is infeasible. 
Although the feasible set of many important problems has an easy-to-handle structure, in general this set could be so complex that the exact projection can not be easily computed. It is well known that the mostly computational burden of each iteration of the GPM is in the solution of this subproblem. In fact, one drawback of methods that use exact projections is to solve a quadratic problem at each stage, which can lead to a substantial increase in the cost per iteration if the number of unknowns is large. In order to reduce the computational effort spent on projections, inexact procedures have been proposed, resulting in more efficient methods, see for example \cite{BirgMartRay2000, Fan_Wang_Yan2019, Patrascu_Necoara2018, Zhang_Wang_Yang2019}. Moreover, considering inexact schemes provides theoretical support for real computational implementations of exact methods. It is worth mentioning that throughout the years there has been an increase in the popularity of inexact methods due to the emergence of large-scale problems in compressed sensing, machine learning applications and data fitting, see for instance \cite{Golbabaee_Davies2018, Mark_Nicolas_Francis2011, Man_Cho_Zirui2017, Sra_Nowozin_Wright2012}.
Motivated by practical and theoretical reasons, the purpose of the present  paper is to present a  new inexact version of the GPM, which we call \textit{Gradient-InexP method} (GInexPM). It consists of using a general inexact projection instead of the exact one used in the GPM. The inexact projection concept considered in the present paper is a variation of the one appeared in \cite[Example 1]{VillaSalzo2013}, which is defined by using an approximated property of the exact projection. 
In particular, it accepts the exact projection which can be adopted when it is easily obtained (for instance, the exact projection onto a box constraint or Lorentz cone can be easily obtained; see \cite[p. 520]{NocedalWright2006} and \cite[Proposition~3.3]{FukushimaTseng2002}, respectively). It is worth noting that our approach to compute the inexact projection has not being considered in the study of the classical gradient method, in particular it is different from the ones proposed in \cite{BirgMartRay2000, Fan_Wang_Yan2019, Golbabaee_Davies2018, Patrascu_Necoara2018, Man_Cho_Zirui2017, Zhang_Wang_Yang2019}. The analyses of  GInexPM will be made employing two diferent step sizes, namely, constant step size and Armijo's step size along the feasible direction. We point out that these step sizes are discussed extensively in the literature on this subject, where many of our results were inspired, see for example \cite{yunier_roman2010, Bertsekas1999, Iusem2003, IusemSvaiter95, Lee_Wright2019, Nesterov2004}. 

\textit{Contributions:} The main novelty in our work is the use of relative error tolerances in the computation of the inexact projection, to analyze the convergence properties of GInexPM.  Our numerical experiments showed that GInexPM outperformed the GPM on a set of least squares problem over the spectrohedron. From a theoretical point of view, under suitable assumptions, the classic results of GPM were obtained for GInexPM as well. More specifically, we have showed that all cluster points of the  sequence generated by GInexPM with constant step size or Armijo's step size are solutions of problem~\eqref{eq:OptP}. Futhermore, under convexity of the objective function, this sequence converges to a solution, if any. In both cases, the analysis establishes convergence results without any compactness assumption. We have also studied iteration-complexity bounds of GInexPM for both  constant step size and Armijo's step size. The presented analysis establishes that the complexity bound  $\mathcal{O}(1/\sqrt{k})$ is unveil for finding $\epsilon$-stationary points for Problem \eqref{eq:OptP}, and, under convexity on $f$, the rate to find a $\epsilon$-optimal functional value is $\mathcal{O}(1/k)$.

\textit{Content of the paper:} In section \ref{Sec:Prel}, some notations and basic results used throughout the paper is presented. In particular, section \ref{Sec:SubInexProj} is devoted to present the concept relative feasible inexact projection and some properties about this concept. In section \ref{SubSec:CSR}, we describe GInexPM method using the constant step size. The results of convergence using constant step size, as well as, results of iteration-complexity bound are presented in the sections \ref{SubSec:CAnalysisC} and \ref{SubSec:IntComp}, respectively. The results related to Armijo’s step sizes is presented in section \ref{SubSec:Armijo}. In section \ref{SubSec:CAnalysisA},  we present the asymptotic convergence analysis of GInexPM using Armijo's step size, and an iteration-complexity bound is presented in section \ref{SubSec:IterCompArm}. Numerical experiments are provided in section \ref{Sec:NumExp}. Finally, the last section presents some final considerations.\\

\section{Preliminaries}  \label{Sec:Prel}
In this section, we introduce  some notation and results used throughout our presentation. We denote ${\mathbb{N}}:=\{0,1,2,\ldots\}$,  $\langle \cdot,\cdot \rangle$ is the usual inner product in $\mathbb{R}^n$ and $\|\cdot\|$ is the Euclidean norm. 
 Let  $f:  {\mathbb R}^n \to {\mathbb R}$ be a differentiable function and $C \subseteq {\mathbb R}^n$. The  gradient $\nabla f$ of $f$ is said to be Lipschitz continuous on $C$ with constant $L>0$ if
\begin{equation} \label{eq:LipCond}
	\|\nabla f(x)-\nabla f(y)\|\leq L \|x-y\|, \qquad \forall~x, y\in C.
\end{equation}
Combining this definition with the fundamental theorem of calculus, we obtain the following result, for which the proof can found in \cite[Proposition A.24]{Bertsekas1999}.
\begin{lemma} \label{Le:derivlipsch}
	Let $f: {\mathbb R}^n \to {\mathbb R}$ be a differentiable function with   Lipschitz continuous  gradient on $C \subseteq {\mathbb R}^n$ with constant $L>0$. Then, $f(y) - f(x) - \langle \nabla f(x), y-x \rangle \leq \frac{L}{2} \|x-y\|^2$,  for all $x, y\in C$.
\end{lemma}
Let $f: {\mathbb R}^n \to {\mathbb R}$ be a differentiable function and $C \subseteq {\mathbb R}^n$ be a convex set.  The function $f$ is strongly convex on $C$, if there exists a constant $\mu \geq 0$ such that 
\begin{equation} \label{eq:fstrong}
	f(y) - f(x) - \langle \nabla f(x), y-x \rangle \geq \frac{\mu}{2} \|x-y\|^2, \qquad  \forall~ x, y\in C.
\end{equation}
When $\mu = 0$, $f$ is said to be  convex. If $f(x) \leq f(y)$ implies $\langle \nabla f(y), x-y \rangle \leq 0$, for any $x, y \in C$, then $f$  is said to be quasiconvex. Moreover,  $f$ is said to be  pseudoconvex if  $\langle \nabla f(y), x-y \rangle \geq 0$ implies $f(x) \leq f(y)$, for any $x, y \in C$, for more details, see \cite{mangasarian1994}. Recall  that the convexity of a function guarantees pseudoconvexity, which in turn guarantees quasiconvexity, see \cite{mangasarian1994}.

A point ${\bar{x}} \in C$ is  said to be a {\it stationary point} for Problem \eqref{eq:OptP} if 
\begin{equation} \label{eq:StatPoint}
	\langle \nabla f({\bar{x}}),  x-{\bar{x}}\rangle \geq 0, \qquad \forall ~ x\in  C.
\end{equation}
We end this section with a useful concept in the analysis of the sequence generated by the gradient method, for more details, see \cite{burachik1995full}.
\begin{definition} \label{def:QuasiFejer}
	A sequence $(y^k)_{k\in\mathbb{N}}$ in $\mathbb{R}^n$ is quasi-Fej\'er convergent to a set $W\subset {\mathbb R}^n$ if, for every $w\in W$, there exists a sequence $(\epsilon_k)_{k\in\mathbb{N}}\subset\mathbb{R}$ such that $\epsilon_k\geq 0$, $\sum_{k\in \mathbb{N}}\epsilon_k<+\infty$, and $\|y_{k+1}-w\|^2\leq \|y_k-w\|^2+\epsilon_k$,  for all $k\in \mathbb{N}$.
\end{definition}
The main property of a quasi-Fej\'er convergent sequence is stated in the next result, and its proof can be found in \cite{burachik1995full}.
\begin{theorem}\label{teo.qf}
	Let $(y^k)_{k\in\mathbb{N}}$ be a sequence in $\mathbb{R}^n$. If $(y^k)_{k\in\mathbb{N}}$ is quasi-Fej\'er convergent to a nomempty set $W\subset  {\mathbb R}^n$, then $(y^k)_{k\in\mathbb{N}}$ is bounded. Furthermore, if a cluster point $y$ of $(y^k)_{k\in\mathbb{N}}$ belongs to $W$, then $\lim_{k\rightarrow\infty}y_k=y$.
\end{theorem}
\subsection{Inexact projection} \label{Sec:SubInexProj}
In this section, we present the concept of feasible inexact projection onto a closed and convex set. This concept has already been used in \cite{Ademir_Orizon_Leandro2020, OrizonFabianaGilson2018, Reiner_Orizon_Leandro2019}. We also present some new properties of the feasible inexact projection used throughout the paper. The definition of feasible inexact projection is as follows.
\begin{definition} \label{def:InexactProj}
	Let $C\subset {\mathbb R}^n$ be a closed convex set and $\varphi_{\gamma}: {\mathbb R}^n\times {\mathbb R}^n\times {\mathbb R}^n \to {\mathbb R}_{+}$ be a function  satisfying the following condition 
	\begin{equation} \label{eq:phi} 
		\varphi_{\gamma}(u, v, w) \leq \gamma_1 \|v-u\|^2 + \gamma_2 \|w-v\|^2 + \gamma_3 \|w-u\|^2, \qquad \forall~ u, v, w \in {\mathbb R}^n,
	\end{equation}
	where $\gamma = (\gamma_1, \gamma_2, \gamma_3) \in {\mathbb R}^3_+$ is a given forcing parameter. 
 The {\it feasible inexact projection mapping} relative to $u \in C$ with error tolerance $\varphi_{\gamma}$, denoted by ${\cal P}_C(\varphi_{\gamma},u, \cdot): {\mathbb R}^n \rightrightarrows C$, is the set-valued mapping defined as follows
	\begin{equation} \label{eq:Projw} 
		{\cal P}_C(\varphi_{\gamma}, u, v) := \left\{w\in C:~\big\langle v-w, y-w \big\rangle \leq \varphi_{\gamma}(u, v, w), \quad \forall~ y \in C \right\}.
	\end{equation}
	Each point $w\in {\cal P}_C(\varphi_{\gamma}, u, v)$ is called a {\it feasible inexact projection of $v$ onto $C$ relative to $u$  with error tolerance $\varphi_{\gamma}$}.
\end{definition}

The feasible inexact projection  generalizes  the concept usual projection. In the following,  we present some remarks about this concept and some examples of functions satisfying \eqref{eq:phi}. 
\begin{remark}\label{rem: welldef}
Let $\gamma_1$, $\gamma_2$ and $\gamma_3$ be nonnegative forcing parameters, $C\subset {\mathbb R}^n$, $u\in C$ and $\varphi_{\gamma}$ be as in Definition~\ref{def:InexactProj}.  Therefore, for all $v\in {\mathbb R}^n$, it follows from \eqref{eq:Projw} that ${\cal P}_C(\varphi_0, u, v)$ is the exact projection of $v$ onto $C$; see \cite[Proposition~2.1.3, p. 201]{Bertsekas1999}. Moreover, ${\cal P}_C(\varphi_0, u, v) \in {\cal P}_C(\varphi_{\gamma}, u, v)$  which implies that  ${\cal P}_C(\varphi_{\gamma}, u, v)\neq \varnothing$, for all $u\in C$ and $v\in {\mathbb R}^n$. Consequently, the set-valued mapping ${\cal P}_C(\varphi_{\gamma}, u, \cdot)$ as stated in \eqref{eq:Projw} is well-defined. Note that  the following functions $\varphi_{1}(u, v, w) = \gamma_1 \|v-u\|^2 + \gamma_2 \|w-v\|^2 + \gamma_3 \|w-u\|^2$, $\varphi_{2}(u, v, w) = \gamma_1 \|v-u\|^2$, $\varphi_{3}(u, v, w) = \gamma_2 \|w-v\|^2$, $\varphi_{4}(u, v, w) = \gamma_3 \|w-u\|^2$, and $\varphi_{5}(u, v, w)= \gamma_1 \gamma_2 \gamma_3 \|v-u\|^2 \|w-v\|^2 \|w-u\|^2$ satisfy \eqref{eq:phi}.
\end{remark}

Item (a) of the next lemma is a variation of \cite[Lemma 6]{Reiner_Orizon_Leandro2019}. By using item (a), we will derive an inequality that together with this item will play an important role in the remainder of this paper. 
\begin{lemma} \label{pr:cond}
	Let $v \in {\mathbb R}^n$, $u \in C$, $\gamma = (\gamma_1, \gamma_2, \gamma_3) \in {\mathbb R}^3_+$ and $w\in {\cal P}_C(\varphi_{\gamma}, u, v)$. Then, there hold:
	\begin{itemize}
	\item[(a)] $\displaystyle \|w-x\|^2 \leq \|v-x\|^2 + \frac{2\gamma_1+2\gamma_3}{1-2\gamma_3}\|v-u\|^2- \frac{1-2\gamma_2}{1-2\gamma_3}\|w-v\|^2$, for all $x \in C$ and $0 \leq \gamma_3 <1/2$;
	\item[(b)] $\displaystyle \big\langle v-w, y-w \big\rangle \leq \frac{\gamma_1 + \gamma_2}{1-2\gamma_2}\|v-u\|^2+\frac{\gamma_3-\gamma_2}{1-2\gamma_2}\|w-u\|^2$, for all $ y\in C$ and $0 \leq \gamma_2 <1/2$.
	\end{itemize}
\end{lemma}

\begin{proof}
		Let $x \in C$ and $0 \leq \gamma_3 <1/2$. First note that $\|w-x\|^2 = \|v-x\|^2 - \|w-v\|^2 + 2 \langle v-w, x-w \rangle$. Since $w \in {\cal P}_C(\varphi_{\gamma}, u, v)$, combining the last equality with \eqref{eq:phi} and \eqref{eq:Projw}, we obtain
	\begin{equation} \label{eq:fg}
		\|w-x\|^2 \leq \|v-x\|^2 - (1-2\gamma_2)\|w-v\|^2 + 2\gamma_1 \|v-u\|^2 + 2\gamma_3 \|w-u\|^2.
	\end{equation}
	On the other hand, we  have $\|w-u\|^2 = \|v-u\|^2 - \|w-v\|^2 +  2 \langle v-w, u-w \rangle$. Thus, due to $w\in {\cal P}_C(\varphi_{\gamma}, u, v)$ and $u \in C$, using \eqref{eq:phi} and \eqref{eq:Projw}, and considering that $0 \leq \gamma_3 < 1/2$, we have
	$$
	\|w-u\|^2 \leq \frac{1+2\gamma_1}{1-2\gamma_3}\|v-u\|^2 - \frac{1-2\gamma_2}{1-2\gamma_3} \|w-v\|^2.
	$$
	Therefore, combining the last inequality with \eqref{eq:fg}, we obtain the inequality of item $(a)$. For proving item $(b)$, take $y\in C$ and $0 \leq \gamma_2 <1/2$. Using \eqref{eq:phi} and \eqref{eq:Projw}, we have 
	\begin{equation} \label{eq:srebt}
		\big\langle v-w, y-w \big\rangle \leq \gamma_1 \|v-u\|^2 + \gamma_2 \|w-v\|^2 + \gamma_3 \|w-u\|^2.
	\end{equation}
	Applying item $(a)$ with $x=u$, after some algebraic manipulations, we conclude that
	$$
 	\|w-v\|^2 \leq \frac{1+2\gamma_1}{1-2\gamma_2}\|v-u\|^2-\frac{1-2\gamma_3}{1-2\gamma_2}\|w-u\|^2.
	$$
	The last inequality together \eqref{eq:srebt} yield
	$$
	\big\langle v-w, y-w \big\rangle \leq  \left(\gamma_1 + \frac{\gamma_2+2\gamma_1 \gamma_2}{1-2\gamma_2}\right)\|v-u\|^2 + \left(\gamma_3 - \frac{\gamma_2-2\gamma_2\gamma_3}{1-2\gamma_2}\right) \|w-u\|^2, 
	$$
	which is equivalent to the inequality in $(b)$.
\end{proof}
\section{GInexPM employing the constant step size rule} \label{SubSec:CSR}
In this section, we describe the GInexPM with a feasible inexact projection for solving  problem~\eqref{eq:OptP}. The rule for choosing the step size will be the same used in \cite{Beck2014, Bertsekas1999}, namely, the constant step size rule. For that, we take a exogenous sequence of real numbers $(a_k)_{k\in\mathbb{N}}$  satisfying 
\begin{equation} \label{eq:akbk}
	0 \leq a_{k}\leq b_{k-1}-b_{k}, \qquad k=0, 1, \ldots, 
\end{equation}
for some given nonincreasing sequence of nonegative real numbers $(b_k)_{k\in\mathbb{N}}$ converging to zero, {\it with the notation $ b_{-1}\in {\mathbb R}_{++}$ such that $ b_{-1}>b_0$}.

\begin{remark} \label{re:csab}
	Condition \eqref{eq:akbk} implies that $\sum_{k\in \mathbb{N}} a_k < b_{-1}$. Examples of sequences $(a_k)_{k\in\mathbb{N}}$ and $(b_k)_{k\in\mathbb{N}}$ satisfying \eqref{eq:akbk} are obtained  by taking $a_{k}:=b_{k-1}-b_{k}$ and, for a given $\bar{b}>0$: (i) $ b_{-1}=3 \bar{b}$, $b_0=2\bar{b}$, $b_k=\bar{b}/k$, for all $k=1, 2, \ldots$; (ii) $ b_{-1}=3\bar{b}$, $b_0=2\bar{b}$, $b_k=\bar{b}/\ln(k+1)$, for all $k=1, 2, \ldots$.
\end{remark}
The conceptual GInexPM is formally stated  as follows.

\begin{algorithm} \caption{GInexPM employing constant step size}\label{Alg:INP}
	\begin{algorithmic}[h]
		\State \textbf{Step 0:} Take $(a_k)_{k\in\mathbb{N}}$, $(b_k)_{k\in\mathbb{N}}$ satisfying \eqref{eq:akbk} and an error tolerance function  $\varphi_{\gamma}$. Let $x^0\in C$ and set $k=0$.
		\State \textbf{Step 1:} If $\nabla f(x^{k}) = 0$, then {\bf stop}; otherwise, choose real numbers $\gamma_1^k, \gamma_2^k$ and $\gamma_3^k$ such that  
				\begin{equation} \label{eq:Tolerance}
					0 \leq \gamma_1^k + \gamma_2^k \leq\frac{a_k}{\|\nabla f(x^{k})\|^2}, \qquad 0 \leq \gamma_2^k < \bar{\gamma_2} < \frac{1}{2}, \qquad 0 \leq \gamma_3^k < \bar{\gamma} < \frac{1}{2},
				\end{equation}
				and a fixed step size $\alpha >0$ and define the next iterate $x^{k+1}$ as any feasible inexact projection of $z^k := x^{k}-\alpha \nabla f(x^{k})$ onto $C$ relative to $x^{k}$ with   forcing parameters $\gamma^k:= (\gamma^k_1, \gamma^k_2, \gamma^k_3)$, i.e.,
				\begin{equation*} \label{eq:IntStep}
					x^{k+1} \in {\cal P}_C\left(\varphi_{\gamma^k}, x^{k}, z^k \right).
				\end{equation*}
		\State \textbf{Step 2:} Set $k\gets k+1$, and go to \textbf{Step~1}.
	\end{algorithmic}
\end{algorithm}

Let us describe the main features of the GInexPM. Firstly, we take exogenous sequences $(a_k)_{k\in\mathbb{N}}$ and $(b_k)_{k\in\mathbb{N}}$ satisfying \eqref{eq:akbk} and an error tolerance function  $\varphi_{\gamma}$. Then, we check if at the current iterate $x^{k}$ we have $\nabla f(x^{k}) = 0$, otherwise, we choose nonnegative forcing parameters $\gamma_1^k$, $\gamma_2^k$ and $\gamma_3^k$ satisfying \eqref{eq:Tolerance}. Set a fixed step size $\alpha >0$. By using some inner procedure, the next iterate $x^{k+1}$ is computed as any feasible inexact projection of $z^k = x^k - \alpha\nabla f(x^k)$ onto the feasible set $C$ relative to $x^k$, i.e., $x^{k+1}\in {\cal P}_C(\varphi_{\gamma^k}, x^{k}, z^k)$; an example of such procedure will be presented in section~\ref{Sec:NumExp}. Note that, if $\gamma_1^k \equiv 0$, $\gamma_2^k \equiv 0$ and $\gamma_3^k \equiv 0$, then ${\cal P}_C({\varphi_0}, x^{k}, z^k)$ is the exact projection, see Remark~\ref{rem: welldef}, and our method corresponds to the usual projected gradient method proposed, for example, in \cite{Beck2014, Bertsekas1999}. It is worth noting that $\gamma_1^k$ and $\gamma_2^k$ in \eqref{eq:Tolerance} can be chosen as any nonnegative real numbers satisfying $0 \leq (\gamma_1^k + \gamma_2^k) \|\nabla f(x^{k})\|^2 \leq a_k$, for prefixed sequences $(a_k)_{k\in\mathbb{N}}$ and $(b_k)_{k\in\mathbb{N}}$ satisfying \eqref{eq:akbk}. In this case, we have
\begin{equation} \label{eq:ToleranceCond}
	\sum_{k \in \mathbb{N}} \big[(\gamma_1^k + \gamma_2^k)\|\nabla f(x^k)\|^2\big] < +\infty.
\end{equation}
In the next sections, we will deal with the convergence analysis of the sequence $(x^k)_{k\in\mathbb{N}}$ generated by GInexPM.
\subsection{Asymptotic convergence analysis} \label{SubSec:CAnalysisC}
The aim of this section is to prove the main convergence results about the asymptotic behavior   of the sequence $(x^k)_{k\in\mathbb{N}}$ generated by Algorithm~\ref{Alg:INP}. We assume that the {\it gradient of the objective function $f$ is Lipschitz continuous} with constant $L>0$. Moreover, we also assume that
\begin{equation}\label{eq:alpha}
	0< \alpha < \frac{1-2 \bar{\gamma}}{L}.
\end{equation}
For future references, it is convenient to  define the following constants:
\begin{equation}\label{eq:NuRho}
	\nu := \frac{1-\bar{\gamma_2}-\bar{\gamma}}{\alpha} - \frac{L}{2} > 0, \qquad \rho := \frac{\alpha}{1-2 \bar{\gamma_2}} > 0.
\end{equation}
In the sequel, we state and prove our first result for the sequence $(x^k)_{k\in\mathbb{N}}$.  The obtained inequality is the counterpart of the one obtained, for example, in \cite[Lemma 9.11, p. 176]{Beck2014}.  
\begin{lemma} \label{Le:MainConvP}
The following inequality holds:
		\begin{equation} \label{eq:MainIneqP}
		f(x^{k+1}) \leq f(x^{k}) + \rho (\gamma_1^k + \gamma_2^k) \|\nabla f(x^k)\|^2 - \nu \|x^{k+1}-x^{k}\|^2, \qquad \forall~{k\in \mathbb{N}}. 
	\end{equation} 
\end{lemma}

\begin{proof}
Since $\nabla f$ satisfies \eqref{eq:LipCond}, applying Lemma \ref{Le:derivlipsch} with $x=x^k$ and $y=x^{k+1}$, we obtain   
	\begin{equation*} \label{eq:fxk6}
		f(x^{k+1}) \leq f(x^{k}) + \big\langle\nabla f(x^{k}), x^{k+1}-x^{k}\big\rangle +\frac{L}{2}\|x^{k+1}-x^{k}\|^2.
	\end{equation*}
Thus, after some algebraic manipulations, we have
	\begin{equation} \label{eq:filc}
	f(x^{k+1}) \leq f(x^{k}) + \frac{1}{\alpha}\big\langle[x^k-\alpha \nabla f(x^k)]-x^{k+1}, x^{k}-x^{k+1}\big\rangle - \left(\frac{1}{\alpha}-\frac{ L}{2}\right) \|x^{k+1}-x^{k}\|^2.
	\end{equation}
Since $x^{k+1} \in {\cal P}_C(\varphi_{\gamma^k}, x^{k}, z^k)$ with $z^k = x^{k}-\alpha \nabla f(x^{k})$, applying item~$(b)$ of Lemma~\ref{pr:cond} with $u=x^k$,  $y=x^k$, $v=z^k$, $w=x^{k+1}$, and $\varphi_{\gamma} = \varphi_{\gamma_k}$, we have
	$$
	\big\langle[x^k-\alpha \nabla f(x^k)]-x^{k+1}, x^{k}-x^{k+1}\big\rangle \leq \frac{\gamma_1^k+\gamma_2^k}{1-2\gamma_2^k} \alpha^2\|\nabla f(x^{k})\|^2 + \frac{\gamma_3^k - \gamma_2^k}{1-2\gamma_2^k} \|x^{k+1} - x^{k}\|^2. 
	$$
Then, combining \eqref{eq:filc} with  the latter inequality yields 
	$$
	f(x^{k+1}) \leq f(x^{k}) + \frac{\alpha(\gamma_1^k + \gamma_2^k)}{1-2\gamma_2^k} \|\nabla f(x^k)\|^2 - \left[\frac{1-\gamma_2^k - \gamma_3^k}{\alpha (1-2\gamma_2^k)} - \frac{L}{2}\right] \|x^{k+1}-x^{k}\|^2.
	$$
Therefore, taking into account \eqref{eq:Tolerance} and \eqref{eq:NuRho}, we have \eqref{eq:MainIneqP}, which concludes the proof.
\end{proof}
The next result is an immediate consequence of Lemma~\ref{Le:MainConvP}.
\begin{corollary} \label{cor:fmonot}
	The sequence $(f(x^{k}) +\rho b_{k-1})_{k\in\mathbb{N}}$ is monotone non-increasing. In particular, $\inf_{k}(f(x^{k}) + \rho b_{k-1})= \inf_{k}f(x^{k})$.
\end{corollary}
\begin{proof}
	Combining \eqref{eq:Tolerance} with \eqref{eq:MainIneqP}, we have  $f(x^{k+1}) \leq f(x^{k}) +\rho a_k$, for all ${k\in \mathbb{N}}$. Thus, taking into account that $(a_k)_{k\in\mathbb{N}}$ satisfies \eqref{eq:akbk}, we obtain   $f(x^{k+1}) + \rho{b_{k}} \leq f(x^{k}) + \rho{b_{k-1}}$, for all ${k\in \mathbb{N}}$,  implying the first statement. Since $(b_k)_{k\in\mathbb{N}}$ converges to zero, the second statement holds.
\end{proof}
Now, we are ready to state and prove a partial asymptotic convergence result on $(x^k)_{k\in\mathbb{N}}$.
\begin{theorem}\label{teo.Main}
	Assume that $-\infty <f^*$. If ${\bar x}\in C$ is a cluster point of the sequence $(x^k)_{k\in\mathbb{N}}$, then ${\bar x}$ is a stationary point for problem \eqref{eq:OptP}. 
\end{theorem}
\begin{proof}
By \eqref{eq:Tolerance}, we have $(\gamma_1^k + \gamma_2^k)\|\nabla f(x^k)\|^2\leq a_k$, for all ~${k\in \mathbb{N}}$. Thus,    Lemma~\ref{Le:MainConvP}  implies that that $\nu\|x^{k+1}-x^{k}\|^2 \leq [f(x^{k}) - f(x^{k+1})] +\rho a_k$, for all ~${k\in \mathbb{N}}$. Using \eqref{eq:akbk}, after some adjustments, we obtain $\|x^{k+1}-x^{k}\|^2 \leq \left[f(x^{k}) + \rho b_{k-1} \right]/{\nu}- \left[ f(x^{k+1}) + \rho b_{k} \right]/{\nu}$, for all $k\in \mathbb{N}$. Hence, due to   $f^*\leq \inf_{k}f(x^{k})$,  Corollary~\ref{cor:fmonot} implies $\sum_{\ell=0}^k\|x^{\ell+1}-x^{\ell}\|^2 \leq [f(x^{0}) + \rho b_{-1}]/{\nu}-  f^*/{\nu}$. Thus, we conclude that $\lim_{k\to +\infty}\|x^{k+1}-x^{k}\|=0$. Let ${\bar x}$ be a cluster point of $(x^k)_{k\in\mathbb{N}}$ and $(x^{k_j})_{j\in\mathbb{N}}$ a subsequence of $(x^k)_{k\in\mathbb{N}}$ such that $\lim_{j\to +\infty}x^{k_j}=~\bar{x}$. Since, $\lim_{j\to +\infty}(x^{k_j+1}- x^{k_j})=0$, we have $\lim_{j\to +\infty}x^{k_j+1}={\bar x}$. On the other hand, due to $x^{k_j+1} \in {\cal P}_C(\varphi_{\gamma^{k_j}}, x^{k_j}, z^{k_j})$, where $z^{k_j}:=x^{k_j}-\alpha \nabla f(x^{k_j})$, applying item~$(b)$ of Lemma~\ref{pr:cond} with $v=z^{k_j}$, $u=x^{k_j}$, $w=x^{k_j+1}$, and $\varphi_{\gamma} = \varphi_{\gamma^{k_j}},$ we obtain   
	$$
	\big\langle z^{k_j}-x^{k_j+1}, y-x^{k_j+1} \big\rangle \leq \frac{\gamma_1^{k_j} + \gamma_2^{k_j}}{1-2\gamma_2^{k_j}} \alpha^2\|\nabla f(x^{k_j})\|^2 + \frac{\gamma_3^{k_j}-\gamma_2^{k_j}}{1-2\gamma_2^{k_j}} \|x^{k_j+1} - x^{k_j}\|^2, \qquad \forall~y\in C.
	$$
	Thus, taking limits on both sides of the last  inequality, we conclude, by using \eqref{eq:ToleranceCond} and continuity of $\nabla f$, that $\big\langle[{\bar x}-\alpha \nabla f({\bar x})]-{\bar x}, y-{\bar x}\big\rangle\leq 0$, for all $y\in C$. Therefore, $\big\langle\nabla f({\bar x}), y-{\bar x}\big\rangle\geq 0$, for all $y\in C$, which implies that ${\bar {x}} \in C$ is a stationary point for problem~\eqref{eq:OptP}.
\end{proof}
In the following lemma, we establish a basic inequality satisfied by $(x^k)_{k\in\mathbb{N}}$. In particular, it will be useful to prove the full asymptotic convergence of $(x^k)_{k\in\mathbb{N}}$ under quasiconvexity of $f$.
\begin{lemma} \label{Le:FejerConv}
For each $x \in C$ and ${k\in \mathbb{N}}$, there holds
$$
\|x^{k+1}-x\|^2 \leq \|x^k-x\|^2+ 2\alpha \rho (\gamma_1^k + \gamma_2^k)\|\nabla f(x^{k})\|^2 +  2 \alpha \big[f(x^k) - f(x^{k+1}) + \langle \nabla f(x^k), x-x^k\rangle\big].
$$
\end{lemma}
\begin{proof}
	Let $x\in C$. By using $z_k = x^{k}-\alpha \nabla f(x^{k})$, after some algebraic manipulations,  we have 
	\begin{align*} 
		\|x^{k+1}-x\|^2 = \|x^{k}-x\|^2 - \|x^{k+1}-x^{k}\|^2 + 2\big\langle z^k-x^{k+1},x-x^{k+1} \big\rangle   +  2\alpha \big\langle \nabla f(x^k),x-x^{k+1}\big\rangle.
	\end{align*}
	Since $x^{k+1} \in {\cal P}_C(\varphi_{\gamma^k}, x^{k}, z^{k})$, applying item~$(b)$ of Lemma~\ref{pr:cond} with $y=x$, $v=z^{k}$, $u=x^{k}$, $w=x^{k+1}$, and $\varphi_{\gamma} = \varphi_{\gamma^k},$ we obtain   
	$$
	\big\langle z^{k}-x^{k+1}, x-x^{k+1}\big\rangle \leq \frac{\gamma_1^k + \gamma_2^k}{1-2\gamma_2^k} \alpha^2\|\nabla f(x^{k})\|^2 + \frac{\gamma_3^k - \gamma_2^k}{1-2\gamma_2^k} \|x^{k+1} - x^{k}\|^2.
	$$
	On the other hand, since $\nabla f$ satisfies \eqref{eq:LipCond},  Lemma \ref{Le:derivlipsch} with $x=x^k$ and $y=x^{k+1}$ yields 	
	\begin{align*}
		\big\langle\nabla f(x^k),x-x^{k+1}\big\rangle &= \big\langle\nabla f(x^{k}), x^{k}-x^{k+1}\big\rangle + \big\langle\nabla f(x^k),x-x^{k}\big\rangle \notag\\
                                                          &\leq f(x^{k}) - f(x^{k+1}) + \frac{L}{2}\|x^{k+1}-x^{k}\|^2 + \big\langle\nabla f(x^k),x-x^{k}\big\rangle.                                                                      
	\end{align*}
	Combining last two inequities with the above equality, we conclude that 
	\begin{multline} \label{eq:fimi}
		\|x^{k+1}-x\|^2 \leq \|x^k-x\|^2 - \left[\frac{1-2\gamma_3^k}{1-2\gamma_2^k} - \alpha L \right] \|x^{k+1}-x^{k}\|^2 \\+ 2 \alpha^2\left(\frac{\gamma_1^k+\gamma_2^k}{1-2\gamma_2^k}\right)\|\nabla f(x^{k})\|^2 + 2 \alpha \left[f(x^k) - f(x^{k+1}) + \big\langle\nabla f(x^k), x-x^k \big\rangle\right].
	\end{multline}
	Taking into account \eqref{eq:Tolerance}, we have $0 \leq 1 - 2 \bar{\gamma_2}   \leq  1-2\gamma_2^k \leq 1$ and $1-2\gamma_3^k \geq 1-2 \bar{\gamma} \geq 0$. Hence, it follows from \eqref{eq:alpha} and \eqref{eq:NuRho} that 
	$$
	\frac{\alpha}{1-2\gamma_2^k} \leq \frac{\alpha}{1-2 \bar{\gamma_2}} = \rho, \qquad \qquad   \left[\frac{1-2\gamma_3^k}{1-2\gamma_2^k} - \alpha L \right] \geq \left[1-2\gamma_3^k - \alpha L \right] \geq \left[1-2 \bar{\gamma} - \alpha L \right] \geq 0.
	$$
	These inequalities, together with \eqref{eq:fimi}, imply  the desired inequality, which concludes the proof. 
\end{proof}
To proceed with the analysis of $(x^k)_{k\in\mathbb{N}}$, we also need  the following auxiliary set
\begin{equation*} \label{eq;SetT}
	T := \left\{x \in C: f(x) \leq \inf_{k}f(x^{k}), \quad {k\in \mathbb{N}}\right\}.
\end{equation*}
\begin{corollary} \label{cor:xkquasife}
	Assume that $f$ is a quasiconvex function. If $T \neq \varnothing$, then $(x^k)_{k\in\mathbb{N}}$ converges to a stationary point for problem \eqref{eq:OptP}. 
\end{corollary}
\begin{proof}
	Let $x \in T$. Since $f$ is a quasiconvex function and $f(x) \leq f(x^k)$ for all $k\in\mathbb{N}$, we have $\big\langle \nabla f(x^k), x-x^k\big\rangle \leq~0$, for all $k\in\mathbb{N}$. Thus, applying Lemma \ref{Le:FejerConv}, we conclude that 
	$$
	\|x^{k+1}-x\|^2 \leq \|x^k-x\|^2+ 2\alpha \rho (\gamma_1^k + \gamma_2^k)\|\nabla f(x^{k})\|^2 + \\ 2 \alpha \left[f(x^k) - f(x^{k+1}) \right], \qquad \forall~ {k\in \mathbb{N}}.
	$$
	Thus, using the first condition in \eqref{eq:Tolerance} and considering  that $(a_k)_{k\in\mathbb{N}}$ and $(b_k)_{k\in\mathbb{N}}$ satisfy \eqref{eq:akbk}, the latter inequality implies 
	\begin{equation}\label{eq:xkquasiFejer}
		\|x^{k+1}-x\|^2 \leq \|x^k-x\|^2 + 2 \alpha \left([f(x^{k}) + \rho b_{k-1}] - [f(x^{k+1}) + \rho b_k]\right), \quad \forall ~k \in \mathbb{N}.
	\end{equation}
	On the other hand, performing a sum of \eqref{eq:xkquasiFejer} for $k = 0, 1, \ldots, N$ and using that $x \in T$, we obtain
	\begin{equation} \label{eq:epsK}
		\sum_{k=0}^{N} \left( \big[f(x^{k}) + \rho b_{k-1}\big] - \big[f(x^{k+1}) + \rho b_k\big] \right) \leq f(x^{0}) - f(x) + \rho( b_{-1}-b_N), 
	\end{equation} 
	for any $N \in \mathbb{N}$. Therefore, \eqref{eq:xkquasiFejer} and \eqref{eq:epsK} imply that $(x^k)_{k\in\mathbb{N}}$ is quasi-Fej\'er convergent to $T$. Since by assumption $T \neq \varnothing$, it follows from Theorem \ref{teo.qf} that $(x^k)_{k\in\mathbb{N}}$ is bounded. Let $\bar{x}$ be a cluster point of $(x^k)_{k\in\mathbb{N}}$ and $(x^{k_j})_{j\in\mathbb{N}}$ a subsequence of $(x^k)_{k\in\mathbb{N}}$ such that $\lim_{j \to +\infty} x^{k_j} = \bar{x}$. Considering that $f$ is continuous, it follows from Corollary~\ref{cor:fmonot} that 
	$$
	\inf_{k}f(x^{k})= \inf_{k}\left(f(x^{k}) + \rho b_{k-1}\right)=  \lim_{j \to +\infty} \left(f(x^{k_j}) + \rho b_{k_j-1}\right)=\lim_{j \to +\infty} f(x^{k_j})= f(\bar{x}).
	$$
	Therefore $\bar{x} \in T$. Using again Theorem \ref{teo.qf}, we have that $(x^k)_{k\in\mathbb{N}}$ converges to $\bar{x}$, and the conclusion is obtained from Theorem \ref{teo.Main}.
\end{proof}
In the following, we present an important result, when $(x^k)_{k\in\mathbb{N}}$ has no cluster points. This result has already appeared in several papers studding  gradient method with exact projetion, see for example \cite{yunier_roman2010, KiwielMurty1996}; however, since its proof is very simple and concise, we include it here for the sake of completeness.
\begin{lemma}\label{le:quasi}
	 If $f$ is a quasiconvex function and $(x^k)_{k\in\mathbb{N}}$ has no cluster points then $\Omega^* = \varnothing$, $\lim_{k \to \infty} \|x^k\|= \infty$, and $\lim_{k \to \infty} f(x^k) = \inf \{f(x) : x \in C\}$.
\end{lemma}
\begin{proof}
	Since $(x^k)_{k\in\mathbb{N}}$ has no cluster points, then  $\lim_{k \to \infty} \|x^k\|= \infty$. Assume that problem \eqref{eq:OptP} has an optimum, say $\tilde{x}$, so $f(\tilde{x}) \leq f(x^k)$ for all $k$. Thus, $\tilde{x} \in T$. Using Corollary \ref{cor:xkquasife}, we have that $(x^k)_{k\in\mathbb{N}}$ is convergent, contradicting that $\lim_{k \to \infty} \|x^k\|= \infty$. Therefore, $\Omega^* = \varnothing$. Now, we claim that $\lim_{k \to \infty} f(x^k) = \inf \{f(x) : x \in C\}$. If $\lim_{k \to \infty} f(x^k) = -\infty$, the claim holds. Let $f^* = \inf_{x \in C} f(x)$. By contradiction, suppose that $\lim_{k \to \infty} f(x^k) > f^*$.  Then, there exists $\tilde{x} \in C$ such that $f(\tilde{x}) \leq f(x^k)$ for all $k$. Using Corollary \ref{cor:xkquasife}, we obtain  that  $(x^k)_{k\in\mathbb{N}}$ is convergent, contradicting again $\lim_{k \to \infty} \|x^k\|= \infty$, which conclude the proof.
\end{proof}

Finally, we presented the main convergence result when $f$ is pseudoconvex, which is a version of \cite[Corollary 3]{yunier_roman2010} for our algorithm, see also \cite[Proposition 5]{IusemSvaiter95}.
\begin{theorem} \label{teo:pseudo}
Assume that $f$ is a pseudoconvex function. Then,  $\Omega^* \neq \varnothing$ if and only if $(x^k)_{k\in\mathbb{N}}$ has at least one cluster point. Moreover,  $(x^k)_{k\in\mathbb{N}}$ converges to an optimum point if $\Omega^* \neq \varnothing$;  otherwise, $\lim_{k \to \infty} \|x^k\|= \infty$ and $\lim_{k \to \infty} f(x^k) = \inf \{f(x) : x \in C\}$.
\end{theorem}
\begin{proof}
Note that  pseudoconvex functions are quasiconvex. First assume that $\Omega^* \neq \varnothing$. In this case,  we have also $T \neq \varnothing$.   Thus, using Corollary \ref{cor:xkquasife}, we conclude that $(x^k)_{k\in\mathbb{N}}$ converges to a stationary point  of  problem \eqref{eq:OptP} and, in particular, $(x^k)_{k\in\mathbb{N}}$ has a cluster point.   Considering  that $f$ is pseudoconvex, this point is also an optimum point. Reciprocally, let $\bar{x}$ be a cluster point of $(x^k)_{k\in\mathbb{N}}$ and $(x^{k_j})_{j\in\mathbb{N}}$ a subsequence of $(x^k)_{k\in\mathbb{N}}$ such that $\lim_{j\to +\infty} x^{k_j} = \bar{x}$. Since, by Corollary~\ref{cor:fmonot},  $(f(x^{k}) +\rho b_{k-1})_{k\in\mathbb{N}}$ is monotone non-increasing, using the continuity of $f$, we have
$$ \inf_{k}f(x^{k})= \inf_{k}\left(f(x^{k}) + \rho b_{k-1}\right)=  \lim_{j \to +\infty} \left(f(x^{k_j}) + \rho b_{k_j-1}\right)=\lim_{j \to +\infty} f(x^{k_j})= f(\bar{x}).$$
	Therefore $\bar{x} \in T$. From Corollary~\ref{cor:xkquasife}, we obtain that $(x^k)_{k\in\mathbb{N}}$ converges to a stationary point $\tilde{x}$ of  problem \eqref{eq:OptP}. Thus, by \eqref{eq:StatPoint}, we have $\langle \nabla f(\tilde{x}), x - \tilde{x} \rangle \geq 0$ for all $x \in C$, which by the   pseudo-convexity of  $f$  implies   $f(x) \geq f(\tilde{x})$ for all $x \in C$. Therefore, $\bar{x}\in \Omega^*$ and $\Omega^* \neq \varnothing$. The last part of the theorem follows by  combining   the first one with  Lemma~\ref{le:quasi}.
\end{proof}
\subsection{Iteration-complexity bound} \label{SubSec:IntComp}
In this section, we establish some iteration-complexity bounds for the sequence $(x^k)_{k\in\mathbb{N}}$ generated by Algorithm~\ref{Alg:INP}.  For that, we take $x^* \in \Omega^*\neq  \varnothing $, set $f^*=f(x^*)$, and define the constant
$$
\eta := f(x^{0})-f^* + \rho b_{-1}.
$$ 
In  next result, we do not  assume any assumption on the  convexity of  the objective function.
\begin{theorem} \label{Teo:FCompP}
Let $\nu>0$ as in \eqref{eq:NuRho}. Then, for all $N\in \mathbb{N}$, there holds
\begin{equation} \label{eq:fcr}
	\min\left\{ \|x^{k+1}-x^{k}\|:~ \forall ~{k\in \mathbb{N}} \right\} \leq  \frac{\sqrt{\eta/\nu }}{\sqrt{N+1}}.
\end{equation} 
\end{theorem}
\begin{proof}
 It follows from \eqref{eq:Tolerance} and \eqref{eq:MainIneqP} that $f(x^{k+1}) \leq f(x^{k}) +\rho a_k - \nu \|x^{k+1}-x^{k}\|^2$, for all ${k\in \mathbb{N}}$. Hence using  \eqref{eq:akbk}, we have 
$\|x^{k+1} - x^{k}\|^2 \leq \frac{1}{\nu} \left[\big(f(x^{k}) + \rho b_{k-1}\big) - \big(f(x^{k+1}) + \rho b_k\big)\right]$, for all ~${k\in \mathbb{N}}$. Thus, 
summing both sides  for $k = 0, 1, \ldots, N$ and using  that $f^* \leq f(x^k)$ for all ${k\in \mathbb{N}}$ , we obtain  
$$
\sum_{k=0}^N\|x^{k+1} - x^{k}\|^2 \leq \frac{1}{\nu} \big[f(x^{0}) -f^*+ \rho (b_{-1} - b_N) \big]\leq  \frac{1}{\nu} \big[f(x^{0}) -f^*+ \rho b_{-1} ) \big]=\eta/\nu.
$$
	Therefore, $(N+1) \min\left\{\|x^{k+1}-x^{k}\|^2:~k=0,1, \ldots, N \right\} \leq \eta/\nu$, which implies \eqref{eq:fcr}. 
\end{proof}
In the following, we present an iteration-complexity bound for the sequence $(x^k)_{k\in\mathbb{N}}$, for finding $\epsilon$-stationary points   of function $f$.
\begin{theorem}
 For every $N \in \mathbb{N}$, there holds
$$
\min_{k =0, 1, \ldots, N} \big\langle\nabla f(x^k), x^k-x \big\rangle \leq \left[\frac{1}{2 \alpha}\|x^0-x\|^2 +  \eta\right]  \frac{1}{N+1}, \qquad \forall~ x\in C.
$$
As a consequence, given $\epsilon>0$,  the maximum  number of iterations $N$ necessary for Algorithm~\ref{Alg:INP}  to generate an iterate $x^{\ell}$ such that $\big\langle\nabla f(x^{\ell}),x- x^{\ell} \big\rangle> -\epsilon$, for all $x\in C$,  is $N\geq [\frac{1}{2 \alpha}\|x^0-x\|^2 +  \eta]/ \epsilon-1$.
\end{theorem}

\begin{proof}
Using Lemma \ref{Le:FejerConv}  and \eqref{eq:Tolerance}, we obtain
$$
\big\langle\nabla f(x^k), x^k-x \big\rangle \leq \frac{1}{2 \alpha}\left[\|x^k-x^*\|^2 - \|x^{k+1}-x^*\|^2 \right]+ \left[\rho a_k + f(x^k) - f(x^{k+1})\right], 
$$
 for all ${k\in \mathbb{N}}$. Since  $(a_k)_{k\in\mathbb{N}}$ and  $(b_k)_{k\in\mathbb{N}}$ satisfy   \eqref{eq:akbk},  we conclude that 
	$$
	\big\langle\nabla f(x^k), x^k-x \big\rangle \leq \frac{1}{2 \alpha}\left[\|x^k-x^*\|^2 - \|x^{k+1}-x^*\|^2 \right] + \left[\left(f(x^{k}) + \rho b_{k-1}\right) - \left(f(x^{k+1}) + \rho b_k\right)\right].
 $$
Thus,  summing both sides  for $k = 0, 1, \ldots, N$ and using  that $f^* \leq f(x^k)$ for all ${k\in \mathbb{N}}$, we have 
	$$
	\sum_{k=0}^N \big\langle\nabla f(x^k), x^k-x \big\rangle \leq \frac{1}{2 \alpha}\|x^0-x\|^2 +  \big[f(x^{0}) -f^*+ \rho (b_{-1} -  b_N) \big]=\frac{1}{2 \alpha}\|x^0-x\|^2+\eta, 
	$$
which implies that  $(N+1) \min \left\{\big\langle\nabla f(x^k), x^k-x \big\rangle :~ k = 0, 1, \ldots, N \right\} \leq \frac{1}{2 \alpha}\|x^0-x\|^2 +  \eta,$
obtaining the first statement of the theorem.	The second statement follows trivially from the first one. 
\end{proof}
\subsubsection{Iteration-complexity bound under convexity} \label{SubSec:IntCompC}
Next result presents an iteration-complexity bound for $(f(x^k))_{k\in\mathbb{N}}$ when $f$ is a convex function. Similar bound for unconstrained problems can be found in \cite[Theorem~2.1.14]{Nesterov2004}.
\begin{theorem}
 Assume that $f$ is a convex function. Then, for every $N \in \mathbb{N}$, there holds
	$$\min\left\{f(x^k) - f^* : ~k = 1 \ldots, N\right\} \leq \frac{\|x^0 - x^*\| + 2 \alpha\rho b_{-1}}{2\alpha N}.$$
\end{theorem}

\begin{proof}
	Since $f$ is convex, we have $2\alpha \langle \nabla f(x^k), x^*-x^k \rangle \leq 2 \alpha\left[f(x^*) - f(x^k)\right]$. Thus, using \eqref{eq:Tolerance} and  Lemma \ref{Le:FejerConv} with $x = x^*$, we obtain $2 \alpha [f(x^{k+1}) - f^*] \leq \|x^k - x^*\|^2 - \|x^{k+1} - x^*\|^2 + 2\alpha\rho a_{k}$, for all $k\in \mathbb{N}$.  Performing the sum of the this  inequality for $k = 0, 1, \ldots, N - 1$, we have
	$$2 \alpha \sum_{k=0}^{N-1} \left[f(x^{k+1}) - f^*\right] \leq \|x^0 - x^*\|^2 - \|x^{N} - x^*\|^2 + 2 \alpha\rho \sum_{k=0}^{N-1} a_{k}.$$
	Since $\sum_{k\in \mathbb{N}} a_k < b_{-1}$, we have $2\alpha N \min\left\{f(x^k) - f^* : k =  1 \ldots, N\right\} \leq \|x^0 - x^*\| + 2 \alpha\rho b_{-1},$
	which implies  the desired inequality.
\end{proof}
\subsubsection{Iteration-complexity bound under strong  convexity} \label{SubSec:IntCompSC}
Our next goal is to show an iteration-complexity bound for $\left(f(x^k)\right)_{k\in\mathbb{N}}$ when $f$ is strongly convex. For this purpose, we first present an inequality that is a variation of \cite[Lemma 3.6]{bubeck2015}. 
\begin{lemma} \label{le:scl}
 Assume that $f$ is $\mu-$strongly convex.  Then, for all $k \in \mathbb{N}$, there holds
$$
f(x^{k+1}) - f^* \leq \frac{1}{\alpha} \big\langle x^k-x^{k+1}, x^k-x^* \big\rangle - \nu \|x^{k+1}-x^k\|^2 + \frac{\gamma_1^k + \gamma_2^k}{1-2\gamma_2^k} \alpha \|\nabla f(x^k)\|^2 - \frac{\mu}{2} \|x^k - x^*\|^2.
$$
\end{lemma}

\begin{proof}
Applying Lemma~\ref{Le:derivlipsch} with $x=x^k$ and $y=x^{k+1}$, and  then using  \eqref{eq:fstrong}, we obtain 
	\begin{align} \label{eq:fstli}
		f(x^{k+1}) - f^* &\leq \big\langle\nabla f(x^k), x^{k+1}-x^k \big\rangle + \frac{L}{2} \|x^k-x^{k+1}\|^2 + \big\langle\nabla f(x^k), x^k - x^* \big\rangle - \frac{\mu}{2} \|x^k - x^*\|^2. \nonumber \\
		&= \big\langle\nabla f(x^k), x^{k+1}-x^* \big\rangle + \frac{L}{2} \|x^k-x^{k+1}\|^2 - \frac{\mu}{2} \|x^k - x^*\|^2.
	\end{align}
	On the order hand, due to $z^k = x^{k}-\alpha \nabla f(x^{k})$, after some algebraic manipulations, we have
\begin{equation} \label{eq:siip}
	  \big\langle\nabla f(x^k), x^{k+1}-x^* \big\rangle = \frac{1}{\alpha} \big\langle x^k - x^{k+1},   x^{k+1} - x^* \big\rangle +  \frac{1}{\alpha} \big\langle   z^k-x^{k+1}, x^*- x^{k+1} \big\rangle.
\end{equation}
Since $x^{k+1} \in {\cal P}_C(\varphi_{\gamma^k}, x^{k}, z^k)$, applying item~$(b)$ of Lemma~\ref{pr:cond} with $y=x^*$, $v=z^k$, $u=x^k$, $w=x^{k+1}$, and $\varphi_{\gamma} = \varphi_{\gamma^k}$, we obtain
	\begin{equation}\label{eq:sipl}
	 \big\langle z^k-x^{k+1}, x^*- x^{k+1} \big\rangle \leq \frac{\gamma_1^k + \gamma_2^k}{1-2\gamma_2^k}  \|z^k-x^k\|^2 + \frac{\gamma_3^k - \gamma_2^k}{1-2\gamma_2^k} \|x^{k+1} - x^{k}\|^2.
	\end{equation}
	Taking into account that $\langle x^k - x^{k+1}, x^{k+1} - x^* \rangle = \langle x^k - x^{k+1}, x^k - x^* \rangle - \|x^k - x^{k+1}\|^2$ and $z^k -x^{k}=-\alpha \nabla f(x^{k})$, the combination of \eqref{eq:siip} and \eqref{eq:sipl} yields
\begin{multline*}
 \big\langle \nabla f(x^k), x^{k+1}-x^* \big\rangle \leq \frac{1}{\alpha} \big\langle x^k - x^{k+1}, x^k - x^*\big\rangle - \frac{1}{\alpha} \left(\frac{1-\gamma_3^k-\gamma_2^k}{1-2\gamma_2^k} \right) \|x^{k+1} - x^{k}\|^2  \\+  \frac{\gamma_1^k + \gamma_2^k}{1-2\gamma_2^k} \alpha\|\nabla f(x^{k})\|^2.
\end{multline*}
Therefore, considering that $0\leq \gamma_2^k \leq \bar{\gamma_2}$ and $0 \leq \gamma_3^k \leq \bar{\gamma}$, the desired inequality follows by using the first condition in \eqref{eq:NuRho} and \eqref{eq:fstli}.
\end{proof}
To proceed, we assume that the sequence $(x^k)_{k\in\mathbb{N}}$ converges to a point $x^* \in \Omega^*$. Moreover, to establish the iteration-complexity bound for $(f(x^k))_{k\in\mathbb{N}}$, we also take 
\begin{equation} \label{eq;gto}
	\gamma_1^k = \gamma_2^k=0, \qquad \forall~{k\in \mathbb{N}}.
\end{equation}
\begin{theorem} \label{teo:compxkstr}
 Assume that $f$ is $\mu-$strongly convex on $\mathbb{R}^n$. Then, the following inequality holds
	\begin{equation} \label{eq:limitxk}
		\|x^{k+1}-x^*\|^2 \leq \left(1-\alpha\mu \right) \|x^k - x^*\|^2.
	\end{equation}
\end{theorem}
\begin{proof}
Using Lemma~\ref{le:scl}, considering \eqref{eq;gto} and $f^* \leq f(x^k)$ for all ${k\in \mathbb{N}}$, we have that $-2\big\langle  x^k-x^{k+1},x^k-x^* \big\rangle \leq - \alpha\mu\|x^k - x^*\|^2 -2\alpha\nu \|x^{k+1}-x^k\|^2$. Therefore, since $1-2\alpha\nu<0$ and taking into account  that $\|x^{k+1}-x^*\|^2 = \|x^k-x^*\|^2 + \|x^{k+1}-x^k\|^2 - 2 \big\langle x^k-x^{k+1},x^k-x^*\big\rangle$, we obtain    \eqref{eq:limitxk}. 
\end{proof}

\begin{remark}
Letting   $\alpha = 1/L$, \eqref{eq:limitxk}  yields  $\|x^{k+1}-x^*\|^2 \leq \left(1- \mu/L \right)^{k+1} \|x^0 - x^*\|^2$, which is closed related to \cite[Theorem 3.10]{bubeck2015}. See also \cite[Theorem 2.1.15]{Nesterov2004}, for the unconstrained case.
\end{remark}

\section{GInexPM employing Armijo's step size rule} \label{SubSec:Armijo}
The aim of this section is to present the GInexPM for solving problem \eqref{eq:OptP} employing Armijo's search. Our method is an inexact version of the projected gradient method proposed in \cite{Iusem2003}, see also \cite{yunier_roman2010}. Let us remind the iteration of the projected gradient method: If   the current iterate  $x^k$ is a non-stationary point of problem \eqref{eq:OptP}, then set $z^k=x^k-\alpha_k\nabla f(x^k)$, compute 
$w^k={\cal P}_C(z^k)$ and define  the next iterate as $x^{k+1}=x^k+\tau_k(w^k-x^k)$, where ${\cal P}_C$ is the exact projection operator on $C$, $\alpha_k$ and $\tau_k$ are suitable positive constants. In this scheme, $d^k=w^k-x^k$ is a feasible descent direction for $f$ at $x^k$. Thus, an Armijo's search is employed  to compute   $\tau_k$ so that it decreases the function $f$ at $x^{k+1}$. In the same way as in  Algorithm~\ref{Alg:INP}, we  propose to compute a feasible inexact projection instead of calculating the exact one. Hence, for guarantee that the feasible direction is also a descent direction, we need to use $\varphi_{\gamma}: {\mathbb R}^n\times {\mathbb R}^n\to {\mathbb R}_{+}$ satisfying
\begin{equation*} \label{eq:phiA} 
		\varphi_{\gamma_3}(u, w) \leq \gamma_3 \|w-u\|^2, \qquad \forall~ u,  w \in {\mathbb R}^n,
\end{equation*}
as the error tolerance function, i.e., we take $\gamma_1 = \gamma_2 = 0$  in Definition~\ref{def:InexactProj}. Hence,  {\it the inexact projection ${\cal P}_C(\varphi_{\gamma}, u, v)$ onto $C$ of $v\in {\mathbb R}^n$ relative to $u\in C$ with error tolerance $\varphi_{\gamma}(u, \cdot)$} becomes
\begin{equation} \label{eq:ProjwA} 
		{\cal P}_C(\varphi_{\gamma}, u, v) := \left\{w\in C:~\big\langle v-w, y-w \big\rangle \leq \varphi_{\gamma_3}(u, w), \quad \forall~ y \in C \right\}.
\end{equation}
{\it Also, we assume that the mapping  $(\gamma_3, u, w) \mapsto \varphi_{\gamma_3}(u, w)$ is continuous}. In this case, the gradient algorithm with inexact projection employing Armijo's step size rule is formally defined as follows.
\begin{algorithm}\caption{GInexPM employing Armijo search}\label{Alg:Armijo}
	\begin{algorithmic}[h]
		\State \textbf{Step 0:} Choose $\sigma \in (0, 1)$, $\tau \in (0, 1)$ and $0 < \alpha_{\min} \leq \alpha_{\max}$. Let $x^0\in C$ and set $k=0$.
		\State \textbf{Step 1:} Choose an error tolerance function  $\varphi_{\gamma}$, real numbers $\alpha_k$ and $\gamma_3^k$ such that  
			\begin{equation} \label{eq:TolArm}
				\alpha_{\min}\leq \alpha_k \leq \alpha_{\max}, \qquad \qquad 0 \leq \gamma_3^k \leq \bar{\gamma} < \frac{1}{2},
			\end{equation}
			and take $w^{k}$ as any feasible inexact projection of $z^k := x^{k}-\alpha_k \nabla f(x^{k})$ onto $C$ relative to $x^{k}$ with error tolerance $\varphi_{\gamma_3^k}(x^k, w^k)$, i.e., 
			\begin{equation} \label{eq:PInexArm}
				w^k \in {\cal P}_C\left(\varphi_{\gamma_3^k}, x^{k}, z^k \right).
			\end{equation}
			If $w^k= x^k$, then {\bf stop}; otherwise, set $\tau_k :=\tau^{j_k}$, where 
			\begin{equation}\label{eq:TkArm}
				j_k :=\min \left\{j \in \mathbb{N}:~  f\big(x^{k}+ \tau^{j}(w^k - x^{k})\big) \leq f(x^{k}) + \sigma \tau^{j}\big\langle\nabla f(x^{k}), w^k - x^{k} \big\rangle \right\},
			\end{equation}
			and set the next iterate $x^{k+1}$ as  
			\begin{equation} \label{eq:IterArm}
				x^{k+1} = x^{k} + \tau_k (w^k - x^{k}).
			\end{equation}
		\State \textbf{Step 2:} Set $k\gets k+1$, and go to \textbf{Step~1}.
	\end{algorithmic} 
\end{algorithm}

 Let us describe the main features of Algorithm~\ref{Alg:Armijo}. In Step~1, we check if $w^k= x^k$. In this case, as we will show, the current iterate $x^{k}$ is a solution of problem \eqref{eq:OptP},  otherwise, we choose $\alpha_k$ such that $\alpha_{\min} \leq \alpha_k \leq \alpha_{\max}$. Then, by using some inner procedure, we compute $w^k$ as any feasible inexact projection of $z^k = x_k - \alpha_k\nabla f(x_k)$ onto the feasible set $C$ relative to  $x^k$, i.e., $w^{k}\in {\cal P}_C(\varphi_{\gamma_3^k}, x^{k}, z^k)$. Recall that, if $\gamma_3^k=0$, then ${\cal P}_C(\varphi_{0}, x^{k}, z^k)$ is the exact projection, see Remark~\ref{rem: welldef}. Therefore, Algorithm~\ref{Alg:Armijo}  can be seen as inexact version of the algorithm considered in \cite{yunier_roman2010, Iusem2003}. In the remainder of this section, we study the asymptotic properties and  iteration-complexity bounds related to Algorithm~\ref{Alg:Armijo}. We begin by presenting some import properties of the inexact projection \eqref{eq:ProjwA}.
\begin{lemma} \label{Le:ProjProperty}
	Let $x \in C$, $\alpha > 0$, and $0\leq \gamma_3 \leq \bar{\gamma} < 1/2$. Take $w(\alpha)$ as any feasible inexact projection of $z(\alpha) = x-\alpha \nabla f(x)$ onto $C$ relative to $x$ with error tolerance $\varphi_{\gamma_3}(x, w(\alpha))$, i.e., $w(\alpha) \in {\cal P}_C(\varphi_{\gamma_3}, x, z(\alpha))$. Then, there hold:
	\begin{itemize}
		\item[(i)] $\big\langle \nabla f(x), w(\alpha) - x\big\rangle \leq \left(\dfrac{\gamma_3-1}{\alpha}\right) \|w(\alpha)-x\|^2$;
		\item[(ii)] the point $x$ is stationary for problem \eqref{eq:OptP} if and only if $x \in {\cal P}_C(\varphi_{\gamma_3}, x, z(\alpha))$;
		\item[(iii)] if $x$ is a nonstationary point for problem \eqref{eq:OptP}, then $\big\langle \nabla f(x), w(\alpha) - x \big\rangle < 0$. Equivalently, if there exists ${\bar \alpha}>0$ such that $\big\langle \nabla f(x), w({\bar \alpha}) - x \big\rangle \geq 0$, then $x$ is stationary for problem \eqref{eq:OptP}.  
	\end{itemize}
\end{lemma}

\begin{proof}
	Since $w(\alpha) \in {\cal P}_C(\varphi_{\gamma_3}, x, z(\alpha))$, applying item $(b)$ of Lemma \ref{pr:cond} with $\gamma_1 = \gamma_2 = 0$, $w=w(\alpha)$, $v=z(\alpha)$, $y=x$, and $u=x$, we obtain $\big\langle x-\alpha \nabla f(x)-w(\alpha), x-w(\alpha) \big\rangle \leq \gamma_3 \|w(\alpha)-x\|^2$ which, after some algebraic manipulations, yields the inequality of item $(i)$. For proving item $(ii)$, we first assume that $x$ is stationary for problem \eqref{eq:OptP}. In this case, \eqref{eq:StatPoint} implies that $\big\langle \nabla f(x), w(\alpha)-x \big\rangle \geq 0$. Thus, considering that $\alpha > 0$ and $0 \leq \gamma_3 \leq \bar{\gamma} < 1/2$, the last inequality together item $(i)$ implies that  $w(\alpha) = x$. Therefore, $x \in {\cal P}_C(\varphi_{\gamma_3}, x, z(\alpha))$. Reciprocally, if $x \in {\cal P}_C(\varphi_{\gamma_3}, x, z(\alpha))$ then applying item $(b)$ of Lemma \eqref{pr:cond} with $\gamma_1 = \gamma_2 = 0$, $w = x$, $v = z(\alpha)$, and $u = x$, we obtain $\big\langle  x-\alpha \nabla f(x)-x, y-x \big\rangle \leq 0$, for all $y \in C$. Considering that $\alpha > 0$, the last inequality is equivalently to $\big\langle \nabla f(x), y-x \big\rangle \geq 0$, for all $y \in C$. Thus, according to \eqref{eq:StatPoint}, we conclude that $x$ is stationary point for problem \eqref{eq:OptP}. Finally, for prove item $(iii)$, take $x$ a nonstationary point for problem \eqref{eq:OptP}. Thus  item $(ii)$ implies that $x \notin {\cal P}_C(\varphi_{\gamma_3}, x, z(\alpha))$ and taking into account that $w(\alpha) \in {\cal P}_C(\varphi_{\gamma_3}, x, z(\alpha))$, we conclude that $x \neq w(\alpha)$. Therefore, due to $\alpha > 0$ and $0 < \gamma_3 \leq \bar{\gamma}$, it follows from item $(i)$ that $\big\langle \nabla f(x), w(\alpha) - x \big\rangle < 0$ and the first sentence is proved. Finally, note that the second statement is the contrapositive of the first sentence. 
\end{proof}

The next result follows from item~$(iii)$ of Lemma~\ref{Le:ProjProperty} and its proof will be omitted.  

\begin{proposition}\label{pr:wellstepsize}
	Let $\sigma \in (0, 1)$, $x \in C$, $\alpha > 0$, and $0\leq \lambda < {\bar \lambda}$. Take $w(\alpha)$ as any feasible inexact projection of $z(\alpha) = x-\alpha \nabla f(x)$ onto $C$ relative to $x$ with error tolerance $\varphi_{\lambda}(x, w(\alpha))$, i.e., $w(\alpha) \in {\cal P}_C(\varphi_{\lambda}, x, z(\alpha))$. If $x$ is a nonstationary point for problem \eqref{eq:OptP}, then there exists $\delta > 0$ such that $f\big(x+\zeta [w(\alpha) - x]\big) < f(x) + \sigma \zeta \big\langle \nabla f(x), w(\alpha) - x \big\rangle$, for all $\zeta \in (0, \delta)$.
\end{proposition}

In the following, we establish the well definition of Algorithm \ref{Alg:Armijo}. 

\begin{proposition}\label{pr:welldef}
	The sequence $(x^k)_{k\in\mathbb{N}}$ generated by Algorithm \ref{Alg:Armijo} is well defined and belongs to $C$. 
\end{proposition}

\begin{proof}
	Proceeding by induction, let $x^0 \in C$, $\alpha_{\min} \leq \alpha_0 \leq \alpha_{\max}$ and $0 \leq \gamma_3^0 < \bar{\gamma}$. Set $z^0 = x^{0}-\alpha_0 \nabla f(x^{0})$. Since $C$ is closed and convex, it follows from Remark~\ref{rem: welldef} that ${\cal P}_C(\varphi_{\gamma_3^0}, x^0, z^0)\neq \varnothing$. Thus, we can take $w^0 \in P_C(\varphi_{\gamma_3^k}, x^{0}, z^0)$. If Algorithm \ref{Alg:Armijo} does not stop, i.e., $w^0 \neq x^0$, then it follows from item~$(i)$ of Lemma~\ref{Le:ProjProperty} that $\langle \nabla f(x^0), w^0- x^0 \rangle < 0$. In this case, Propoposition~\ref{pr:wellstepsize} implies that it is possible to compute $\tau_0 \in (0, 1]$ satisfying \eqref{eq:TkArm}, for $k = 0$. Therefore, $x^1 = x^{0} + \tau_0 (w^0 - x^{0})$ in \eqref{eq:IterArm} is well defined and, considering that $x^0, w^0\in C$ and $\tau_0 \in (0, 1]$, we have $x^1 \in C$. The induction step is completely analogous, implying that $(x^k)_{k\in\mathbb{N}}$ is well defined and belongs to $C$. 
\end{proof}

\subsection{Asymptotic convergence analysis} \label{SubSec:CAnalysisA}
The aim of this section is to study asymptotic convergence properties related to Algorithm~\ref{Alg:Armijo}. We begin by presenting a partial convergence result of the sequence $(x^k)_{k\in\mathbb{N}}$ generated by Algorithm~\ref{Alg:Armijo}.  
\begin{proposition} \label{pr:statArm}
Algorithm~\ref{Alg:Armijo} finishes in a finite number of iterations at a stationary point of problem \eqref{eq:OptP}, or generates an infinite sequence $(x^k)_{k\in\mathbb{N}}$ for which $\left(f(x^k)\right)_{k\in\mathbb{N}}$ is a  decreasing sequence and every cluster point of $(x^k)_{k\in\mathbb{N}}$ is stationary for problem~\eqref{eq:OptP}.
\end{proposition}
\begin{proof}
	First we assume that $(x^k)_{k\in\mathbb{N}}$ is finite. In this case, according to Step~1,   there exists $k \in \mathbb{N}$ such that $x^k = w^k \in {\cal P}_C(\varphi_{\gamma_3^k}, x^{k}, z^k)$, where $z^k = x^{k}-\alpha_k \nabla f(x^{k})$, $0 \leq \gamma_3^k \leq \bar{\gamma}$ and $\alpha_k > 0$. Therefore, applying the second statement of item $(ii)$ of Lemma~\ref{Le:ProjProperty} with $x = x^{k}$, $\alpha = \alpha_k$, and $\gamma_3 = \gamma_3^k$, we conclude that $x^k$ is stationary for problem~\eqref{eq:OptP}. Now, we assume that $(x^k)_{k\in\mathbb{N}}$ is infinite. Thus, according to Step~1, $x^k \neq w^k$ for all $k = 0,1, \ldots$. Consequently, applying item $(ii)$ of Lemma~\ref{Le:ProjProperty} with $x = x^{k}$, $\alpha = \alpha_k$, and $\gamma_3 = \gamma_3^k$, we have that $x^k$ is nonstationary for problem~\eqref{eq:OptP}. Hence,  item $(iii)$ of Lemma~\ref{Le:ProjProperty} implies that $\big\langle \nabla f(x^k), w^k- x^k \big\rangle < 0$, for all $k = 0, 1, \ldots$. Therefore, it follows from \eqref{eq:TkArm} and \eqref{eq:IterArm} that 
\begin{equation}\label{eq:fmotArm1}
	0 < -\sigma\tau_{k} \big\langle\nabla f(x^{k}), w^{k}-x^{k} \big\rangle \leq f(x^{k}) - f(x^{k+1}), \qquad \forall ~k \in \mathbb{N}, 
\end{equation}
with implies that $f(x^{k+1}) < f(x^{k})$, for all $k = 0, 1, \ldots$, and then $\left(f(x^k)\right)_{k\in\mathbb{N}}$ is a decreasing sequence. Let ${\bar x}$ be a cluster point of $(x^k)_{k\in\mathbb{N}}$ and $(x^{k_j})_{j\in\mathbb{N}}$ a subsequence of $(x^k)_{k\in\mathbb{N}}$ such that $\lim_{j\to +\infty} x^{k_j} = \bar{x}$. Since $C$ is closed, by Proposition~\ref{pr:welldef}, we have $\bar{x} \in C$.  Since $\left(f(x^k)\right)_{k\in\mathbb{N}}$ is decreasing and $\lim_{j\to +\infty} f(x^{k_j}) =f(\bar{x})$, we conclude that $\lim_{k\to +\infty} f(x^{k}) =f(\bar{x})$.  On the order hand, using the last condition in \eqref{eq:TolArm}, we have $1/(1-2\gamma_3^k) < 1/(1-2 \bar{\gamma})$, for all $k = 0, 1, \ldots$. Since $w^k \in {\cal P}_C(\varphi_{\gamma_3^k}, x^{k}, z^k)$, where $z^k = x^{k}-\alpha_k \nabla f(x^{k})$, applying item~$(a)$ of Lemma~\ref{pr:cond} with $x = x^k$, $u = x^k$, $v = z^k$, $w = w^k$, $\gamma_1 = \gamma_2 = 0$, and $\gamma_3 = \gamma_3^k$, we obtain
\begin{equation*} \label{eq:wkltdArm}
	\|w^{k_j} - x^{k_j} \|^2 \leq \frac{\alpha_{k_j} ^2}{1-2\gamma_3^{k_j}} \|\nabla f(x^{k_j} )\|^2 < \frac{\alpha_{\max}^2}{1-2\bar{\gamma}} \|\nabla f(x^{k_j})\|^2, \qquad \forall ~j \in \mathbb{N}.
\end{equation*}
Considering that $(x^{k_j})_{j\in\mathbb{N}}$ converges to ${\bar x}$ and $\nabla f$ is continuous, the last inequality implies that $(w^{k_j})_{j\in\mathbb{N}}\subset C$ is also bounded. Thus, we can assume without loss of generality that $\lim_{j\to +\infty} w^{k_j} = \bar{w}\in C$. Now, due to $\tau_k \in (0,1]$, for all $k=0,1, \ldots$, we can also assume without loss of generality that $\lim_{j \to +\infty} \tau_{k_j} = \bar{\tau} \in [0,1].$ 
Therefore, owing to $\lim_{j\to +\infty} f(x^{k}) =f(\bar{x})$, taking limits in \eqref{eq:fmotArm1} along  an appropriate subsequence, we obtain
$
\bar{\tau} \big\langle \nabla f(\bar{x}), \bar{w}- \bar{x} \big\rangle=0. 
$
We have two possibilities: $\bar{\tau} > 0$ or $\bar{\tau} = 0$. If $\bar{\tau} > 0$, then  $\big\langle \nabla f(\bar{x}), \bar{w}- \bar{x} \big\rangle = 0.$  Now, we  assume that $\lim_{j \to +\infty} \tau_{k_j} = \bar{\tau} = 0$. In this case, for any fixed $q \in \mathbb{N}$, there exits $j$  such that $\tau_{k_j} < \tau^{q}$. Hence, Armijo's condition \eqref{eq:TkArm} does not hold for $ \tau^{q}$, i.e.,
$f\big(x^{k_j}+\tau^{q} (w^{k_j} - x^{k_j})\big) > f(x^{k_j}) + \sigma \tau^{q} \big\langle\nabla f(x^{k_j}), w^{k_j} - x^{k_j} \big\rangle$, for all $j \in \mathbb{N}.$
Thus, taking limits as $j$ goes to $+\infty$, we obtain 
$f\big(\bar{x}+\tau^{q} (\bar{w}- \bar{x})\big) \geq f(\bar{x}) + \sigma \tau^{q} \big\langle \nabla f(\bar{x}), \bar{w}-\bar{x} \big\rangle,$ which is equivalent to 
$$
\frac{f\big(\bar{x}+\tau^{q} (\bar{w} - \bar{x})\big) - f(\bar{x})}{\tau^{q}} \geq \sigma \big\langle \nabla f(\bar{x}), \bar{w}-\bar{x} \big\rangle.
$$
Since this inequality holds for all $q \in \mathbb{N}$, taking limits as $q$ goes to $+\infty$, we conclude that $\langle \nabla f(\bar{x}), \bar{w}-\bar{x} \rangle \geq \sigma \big\langle \nabla f(\bar{x}), \bar{w}-\bar{x} \big\rangle$. Hence, due to $\sigma \in (0, 1)$, we obtain $\big\langle \nabla f(\bar{x}), \bar{w}-\bar{x} \big\rangle \geq 0$.
We recall that $\langle \nabla f(x^{k_j}), w^{k_j}- x^{k_j} \rangle < 0$, for all $j=0, 1, \ldots$, which taking limits as $j$ goes to $+\infty$ yields $\big\langle \nabla f(\bar{x}), \bar{w}-\bar{x} \big\rangle \leq 0$. Hence, $\big\langle \nabla f(\bar{x}), \bar{w}-\bar{x} \big\rangle = 0$. Therefore, for any of two possibilities, $\bar{\tau} > 0$ or $\bar{\tau} = 0$, we have $\langle \nabla f(\bar{x}), \bar{w}-\bar{x} \rangle = 0$. On the other hand, 
$$w^{k_j} \in {\cal P}_C(\varphi_{\gamma_3^{k_j}}, x^{k_j}, z^{k_j} ),$$ 
where $z^{k_j} = x^{k_j}-\alpha_{k_j} \nabla f(x^{k_j})$, $0 \leq \gamma_3^{k_j} \leq  \bar{\gamma}$, and $\alpha_{k_j} > 0$. Thus, it  follows from \eqref{eq:ProjwA} that  
\begin{equation} \label{eq:kc}
	\big\langle z^{k_j}-w^{k_j}, y-w^{k_j} \big\rangle \leq \varphi_{\gamma_3^{k_j}}(x^{k_j},  w^{k_j}), \qquad y\in C, \quad \forall ~j \in \mathbb{N}.
\end{equation}
Moreover, since $\alpha_k \in [\alpha_{\min}, \alpha_{\max}]$, for all $k = 0, 1, \ldots$, we also assume without loss of generality that $\lim_{j \to +\infty} \alpha_{k_j} = \bar{\alpha} \in [\alpha_{\min}, \alpha_{\max}]$. Thus, taking limits in \eqref{eq:kc} and considering that the mapping $(\gamma_3, u, w) \mapsto \varphi_{\gamma_3}(u, w)$ is continuous, $\lim_{j\to +\infty} x^{k_j} = \bar{x}\in C$, $\lim_{j\to +\infty} w^{k_j} = \bar{w}\in C$, and $\lim_{j \to +\infty} \tau_{k_j} = \bar{\tau} \in [0,1]$, we conclude that  $\big\langle \bar{z}-\bar{w}, y-\bar{w} \big\rangle \leq \varphi_{\bar{\gamma}}(\bar{x}, \bar{w})$, for all $y\in C$,  where $\bar{z} = \bar{x}-{\bar \alpha} \nabla f(\bar{x})$. Hence, it follows from \eqref{eq:ProjwA} that $\bar{w}\in {\cal P}_C\left(\varphi_{\bar{\gamma}}, {\bar x}, {\bar z}\right)$, where $\bar{z} = \bar{x}-{\bar \alpha} \nabla f(\bar{x})$. Therefore, due to $\big\langle \nabla f(\bar{x}), \bar{w}-\bar{x} \big\rangle = 0$, we can apply the second sentence in item~$(iii)$ of Lemma~\ref{Le:ProjProperty} with $x = \bar{x}$, $z({\bar \alpha}) = \bar{z}$, and $w({\bar \alpha}) = \bar{w}$ to conclude that $\bar{x}$ is stationary for problem~\eqref{eq:OptP}.
\end{proof}
Due to Proposition~\ref{pr:statArm}, {\it from now on we assume that the sequence $(x^k)_{k\in\mathbb{N}}$ generated by Algorithm~\ref{Alg:Armijo} is infinite}. The following result establishes a basic inequality satisfied by the iterates of Algorithm~\ref{Alg:Armijo}, which will be used to study its convergence properties. For simplify the notations we define the following constant:
\begin{equation} \label{eq:eta}
	\xi := \dfrac{2 \alpha_{\max}}{\sigma} > 0.
\end{equation} 
\begin{lemma}\label{Le:xkArm}
	For each  $x\in C$, there holds
	\begin{equation}\label{eq:xkArm}
		\|x^{k+1}-x\|^2 \leq \|x^k-x\|^2 + 2\alpha_k\tau_k \big\langle\nabla f(x^k), x-x^k\big\rangle + \xi \left[f(x^k) - f(x^{k+1}) \right], \quad \forall ~k \in \mathbb{N}.
	\end{equation}
\end{lemma}

\begin{proof}
	We know that $\|x^{k+1}-x\|^2 = \|x^k-x\|^2 + \|x^{k+1}-x^k\|^2 - 2 \big\langle x^{k+1}-x^k, x-x^k \big\rangle$, for all $x \in C$ and $k = 0, 1, \ldots$. Thus, using \eqref{eq:IterArm},  we have	
	\begin{equation}\label{eq:xkArm1}
		\|x^{k+1}-x\|^2 = \|x^k-x\|^2 + \tau_k^2\|w^k - x^{k}\|^2 - 2 \tau_k \big\langle w^k - x^{k}, x-x^k \big\rangle, \qquad \forall ~k \in \mathbb{N}. 
	\end{equation}
	On the other hand, by using \eqref{eq:PInexArm} we have  $w^k \in {\cal P}_C(\varphi_{\gamma_3^k}, x^{k}, z^k)$ with $z^k = x^{k}-\alpha_k \nabla f(x^{k})$. Thus, applying item~$(b)$ of Lemma~\ref{pr:cond} with $y = x$, $u = x^k$, $v = z^k$, $w = w^k$, $\gamma_1 = \gamma_2 = 0$, $\gamma_3 = \gamma_3^k$, and $\varphi_{\gamma_3} = \varphi_{\gamma_3^k}$, we obtain $\langle x^k-\alpha_k\nabla f(x^k)-w^k, x-w^k\rangle \leq \gamma_3^k \|w^k - x^{k}\|^2$, for all $k \in \mathbb{N}$.  After some algebraic manipulations in the last inequality, we have 
	$$
	\big\langle w^k-x^k, x-x^k\big\rangle \geq \alpha_k \big\langle\nabla f(x^k), w^k-x \big\rangle + (1-\gamma_3^k) \|w^k-x^k\|^2.
	$$
	Combining the last inequality with \eqref{eq:xkArm1}, we conclude 
	\begin{equation} \label{eq:xkArm3}
		\|x^{k+1}-x\|^2 \leq \|x^k-x\|^2 - \tau_k \big[2(1-\gamma_3^k) - \tau_k \big] \|w^k-x^k\|^2 + 2\tau_k\alpha_k \big\langle\nabla f(x^k), x-w^k\big\rangle.
	\end{equation}
	Since $0 \leq \gamma_3^k < \bar{\gamma} < 1/2$ and $\tau_k \in (0, 1]$, we have $2(1-\gamma_3^k) - \tau_k \geq 1 - 2 \bar{\gamma} > 0$. Thus, \eqref{eq:xkArm3} becomes
	$$
	\|x^{k+1}-x\|^2 \leq \|x^k-x\|^2 + 2\tau_k\alpha_k \big\langle\nabla f(x^k), x-w^k\big\rangle, \qquad \forall ~k \in \mathbb{N}. 
	$$
	Therefore, considering that $\big\langle \nabla f(x^k), x-w^k\big\rangle = \big\langle \nabla f(x^k), x-x^k\big\rangle + \big\langle \nabla f(x^k), x^k-w^k \big\rangle$ and taking into account \eqref{eq:TkArm}, we conclude that
	\begin{equation} \label{eq;ali}
		\|x^{k+1}-x\|^2 \leq \|x^k-x\|^2 + 2\tau_k\alpha_k \big\langle\nabla f(x^k),x-x^k\big\rangle + \frac{2 \alpha_k}{\sigma} \left[f(x^k)-f(x^{k+1})\right], \qquad \forall ~k \in \mathbb{N}. 
	\end{equation}
	Since $0< \alpha_k\leq  \alpha_{\max}$, Proposition~\ref{pr:statArm} implies $\alpha_k\big[f(x^k)-f(x^{k+1})\big] < \alpha_{\max} \big[f(x^k)-f(x^{k+1})\big]$. Therefore, the desired inequality \eqref{eq:xkArm} follows from \eqref{eq;ali} by using \eqref{eq:eta}. 
\end{proof}

For the sequence $(x^k)_{k\in\mathbb{N}}$ generated by Algorithm~\ref{Alg:Armijo}, we define the following auxiliary set:
\begin{equation*}\label{eq:SetTArm}
	U := \left\{x \in C: f(x) \leq \inf_{k}f(x^{k}), \quad {k\in \mathbb{N}}\right\}.
\end{equation*}
Next, we analyze the behavior of the sequence $(x^k)_{k\in\mathbb{N}}$ when $f$ is a quasiconvex function.
\begin{corollary} \label{cor:xkquasifeArm}
	Assume that $f$ is a quasiconvex function. If $U \neq \varnothing$, then $(x^k)_{k\in\mathbb{N}}$ converges to a stationary point of problem~\eqref{eq:OptP}.
\end{corollary}

\begin{proof}
	 Let $x \in U$. Thus, $f(x) \leq f(x^k)$ for all $k\in \mathbb{N}$. Since  $f$ is quasiconvex, we have $\big\langle \nabla f(x^k),x-x^k\big\rangle \leq 0$, for all $k\in \mathbb{N}$. Using Lemma \ref{Le:xkArm}, we obtain
	 $$\|x^{k+1}-x\|^2 \leq \|x^k-x\|^2+ \xi \left[f(x^k) - f(x^{k+1}) \right], \quad \forall~k \in \mathbb{N}.$$
	Defining $\epsilon_k = \xi \big[f(x^k) - f(x^{k+1}) \big]$, we have $\|x^{k+1}-x\|^2 \leq \|x^k-x\|^2+ \epsilon_k$, for all $k \in \mathbb{N}$. On the other hand, summing $\epsilon_k$ with $k = 0, 1, \ldots, N$, we have $\sum_{k=0}^N \epsilon_k \leq \xi \big[f(x^0) - f(x) \big] < \infty$. Thus, it follows from  Definition~\ref{def:QuasiFejer}  that $(x^k)_{k\in\mathbb{N}}$ is quasi-Fejér convergent to $U$. Since  $U$ is nonempty, it follows from Theorem \ref{teo.qf} that $(x^k)_{k\in\mathbb{N}}$ is bounded, and therefore it has a cluster point. Let $\bar{x}$ be a cluster point of $(x^k)_{k\in\mathbb{N}}$ and $(x^{k_j})_{j\in\mathbb{N}}$ be a subsequence of $(x^k)_{k\in\mathbb{N}}$ such that $\lim_{j \to \infty} x^{k_j} = \bar{x}$. Considering that $f$ is continuous, we have $\lim_{j \to \infty} f(x^{k_j}) = f(\bar{x})$.  Hence,  since, by Proposition~\ref{pr:statArm}, $\left(f(x^k)\right)_{k\in\mathbb{N}}$ is decreasing,   we obtain $\inf\{f(x^k) : k = 0, 1, 2, \ldots \} = \lim_{k \to \infty} f(x^k) = f(\bar{x}).$
	Therefore, $\bar{x} \in U$. It follows from Theorem \ref{teo.qf} that $(x^k)_{k\in\mathbb{N}}$ converges to $\bar{x}$ and the conclusion is obtained by  using  again Proposition  \ref{pr:statArm}.
\end{proof}
The next two results are similar  to  Lemma~\ref{le:quasi} and  Theorem~\ref{teo:pseudo}, respectively. For completeness reasons, we have included  their proofs here.
\begin{lemma} \label{le:lfc}
	 If $f$ is a quasiconvex function and $(x^k)_{k\in\mathbb{N}}$ has no cluster points, then $\Omega^* = \varnothing$, $\lim_{k \to \infty} \|x^k\|= \infty$, and $\lim_{k \to \infty} f(x^k) = \inf \{f(x) : x \in C\}$.
\end{lemma}
\begin{proof}
	Since $(x^k)_{k\in\mathbb{N}}$ has no cluster points, then  $\lim_{k \to \infty} \|x^k\|= \infty$. Assume that problem \eqref{eq:OptP} has an optimum, say $\tilde{x}$, so $f(\tilde{x}) \leq f(x^k)$ for all $k$. Thus, $\tilde{x} \in U$. Using Corollary \ref{cor:xkquasifeArm}, we obtain that $(x^k)_{k\in\mathbb{N}}$ is convergent, contradicting that $\lim_{k \to \infty} \|x^k\|= \infty$. Therefore, $\Omega^* = \varnothing$. Now, we claim that $\lim_{k \to \infty} f(x^k) = \inf \{f(x) : x \in C\}$. If $\lim_{k \to \infty} f(x^k) = -\infty$, the claim holds. Let $f^* = \inf_{x \in C} f(x)$. By contradiction, suppose that $\lim_{k \to \infty} f(x^k) > f^*$. Then, there exists $\tilde{x} \in C$ such that $f(\tilde{x}) \leq f(x^k)$ for all $k$. Using Corollary \ref{cor:xkquasifeArm}, we have that  $(x^k)_{k\in\mathbb{N}}$ is convergent, contradicting again $\lim_{k \to \infty} \|x^k\|= \infty$, which concludes the proof.
\end{proof}

\begin{theorem}
Assume that $f$ is a pseudoconvex function. Then,  $\Omega^* \neq \varnothing$ if and only if $(x^k)_{k\in\mathbb{N}}$ has at least one cluster point. Moreover, $(x^k)_{k\in\mathbb{N}}$  converges to an optimum point if $\Omega^* \neq \varnothing$;  otherwise, $\lim_{k \to \infty} \|x^k\|= \infty$ and $\lim_{k \to \infty} f(x^k) = \inf \{f(x) : x \in C\}$.
\end{theorem}
\begin{proof}
Recall that  pseudoconvex functions are also  quasiconvex.  Assume that $\Omega^* \neq \varnothing$. In this case, $U \neq \varnothing$.  Using Corollary~\ref{cor:xkquasifeArm}, we conclude that $(x^k)_{k\in\mathbb{N}}$ converges to a stationary point of problem~\eqref{eq:OptP}.  
Reciprocally, let $\bar{x}$ be a cluster point of $(x^k)_{k\in\mathbb{N}}$ and  $(x^{k_j})_{j\in\mathbb{N}}$ be a subsequence of $(x^k)_{k\in\mathbb{N}}$ such that $\lim_{j\to +\infty} x^{k_j} = \bar{x}$. Since, from Proposition~\ref{pr:statArm},  $(f(x^{k}))_{k\in\mathbb{N}}$ is monotone non-increasing, by continuity of $f$, we conclude that $ \inf_{k}f(x^{k})= f(\bar{x})$ and hence $\bar{x} \in U$. Using  Corollary~\ref{cor:xkquasifeArm}, we obtain that $(x^k)_{k\in\mathbb{N}}$ converges to a stationary point $\tilde{x}$ of  problem \eqref{eq:OptP}.  Since $f$ is   pseudoconvex, this point  is also an optimal solution of problem~\eqref{eq:OptP}. The last part of the theorem follows by  combining   the first one with  Lemma~\ref{le:lfc}.
\end{proof}
\subsection{Iteration-complexity bound}\label{SubSec:IterCompArm}
This section is devoted to study iteration-complexity bounds for the sequence generated by Algorithm~\ref{Alg:Armijo},  similar results in the multiobjetive context can be found in \cite{fliege_vicente2019}. For that we {\it assume that $f^*>-\infty$ and the objective function $f$ has  Lipschitz continuous gradient with constant $L \geq 0$, i.e., we assume that $\nabla f$ satisfies \eqref{eq:LipCond}}. Moreover, we also assume that $(x^k)_{k\in\mathbb{N}}$ generated by Algorithm~\ref{Alg:INP} converges to a point $x^*$, i.e, $\lim_{k \to \infty} x^k = x^*$. To simplify the notations, we set 
\begin{equation} \label{eq;taumin}
 \tau_{\min} := \min \left\{\frac{2\tau(1-\sigma)(1- \bar{\gamma})}{\alpha_{\max} L},~1\right\}.
 \end{equation} 
 The next lemma is a version of \cite[Lemma 3.1]{fliege_vicente2019} for our specific context. 
\begin{lemma}\label{Le:tauminArm}
	The step size $\tau_k$ in Algorithm \ref{Alg:Armijo} satisfies $\tau_k \geq \tau_{\min}$.
\end{lemma}
\begin{proof}
If $\tau_k=1$, then the result trivially holds. Thus,  assume that $\tau_k<1$. It follows from Armijo's condition in \eqref{eq:TkArm}   that $f(x^{k}+ \frac{\tau_k}{\tau } (w^{k} - x^{k})) > f(x^{k}) + \sigma \frac{\tau_k}{\tau }  \big\langle\nabla f(x^{k}), w^{k} - x^{k} \big\rangle$. Now, using  Lemma \ref{Le:derivlipsch}, we have $f\left(x^{k}+ \frac{\tau_k}{\tau } (w^{k} - x^{k})\right) \leq f(x^{k}) + \frac{\tau_k}{\tau } \big\langle\nabla f(x^{k}), w^{k} - x^{k} \big\rangle + \frac{1}{\tau^2 } \frac{L}{2}\tau_k^2 \|w^k-x^k\|^2$. Hence, combining the two previous inequalities,  we obtain
	\begin{equation}\label{eq:ffA3}
		\tau(1-\sigma) \big\langle\nabla f(x^{k}), w^{k} - x^{k} \big\rangle + \frac{L}{2}\tau_k \|w^k-x^k\|^2 > 0.
	\end{equation}
	On the order hand, since $w^k \in {\cal P}_C(\varphi_{\gamma_3^k}, x^{k}, z^k)$, where $z^k = x^{k}-\alpha_k \nabla f(x^{k})$, applying item~$(i)$ of Lemma~\ref{Le:ProjProperty} with $x=x^k$, $w(\alpha) = w^k$, $z = z^k$, $\gamma_1 = \gamma_2 = 0$, $\gamma_3 = \gamma_3^k$, and $\varphi_{\gamma} = \varphi_{\gamma^k}$, we have
	$$
	\big\langle\nabla f(x^{k}), w^k-x^{k}\big\rangle \leq \left(\frac{\gamma_3^k-1}{\alpha_k}\right) \|w^k-x^k\|^2.
	$$
	Combining  the last inequality with \eqref{eq:ffA3} yields  $ \left[{\tau(1-\sigma)(\gamma_3^k-1)}/{\alpha_k} + {L}\tau_k/2 \right]\|w^k-x^k\|^2 > 0. $ Hence, using \eqref{eq:TolArm}, it follows that  
	$$
	\tau_k> \frac{2\tau(1-\sigma)(1-\gamma_3^k)}{\alpha_k L} \geq \frac{2\tau(1-\sigma)(1-\bar{\gamma})}{\alpha_{\max} L}.
	$$
	Therefore, since $\tau_k$ is never larger than one, the result follows and the proof is concluded.
\end{proof}
It follows from item $(ii)$ of Lemma~\ref{Le:ProjProperty} that if $x^k \in {\cal P}_C(\varphi_{\gamma_3}, x^k, z^k)$, then the point $x^k$ is stationary for problem \eqref{eq:OptP}. Since $w^k \in {\cal P}_C(\varphi_{\gamma_3}, x^k, z^k)$, the quantity $\|w^k-x^k\|$ can be seen as a measure of stationarity of $x^k$. Next theorem presents an iteration-complexity bound for this quantity, see a similar result in \cite[Theorem 3.1]{fliege_vicente2019}.
\begin{theorem}
	Let $\tau_{\min}$ be defined in \eqref{eq;taumin}. Then, for every $N \in \mathbb{N}$, the following inequality holds
	$$
	\min\left\{\|w^k-x^k\| :~ k= 0, 1 \ldots, N-1\right\} \leq \sqrt{\frac{\alpha_{\max} \left[f(x^0)-f^*\right] }{\sigma \tau_{\min}{\left(1-\bar{\gamma}\right)}}} \frac{1}{\sqrt{N}}.
	$$
\end{theorem}
\begin{proof}
	From the definition of $\tau_k$ and condition \eqref{eq:TkArm}, we have
	\begin{equation}\label{eq:fD1}
		f(x^{k+1}) - f(x^k) \leq \sigma\tau_k \big\langle\nabla f(x^{k}),  w^k-x^{k} \big\rangle.
	\end{equation}
	Since  $w^k \in {\cal P}_C(\varphi_{\gamma_3^k}, x^{k}, z^k)$, where $z^k = x^{k}-\alpha_k \nabla f(x^{k})$, applying item~$(i)$ of Lemma~\ref{Le:ProjProperty} with $x = x^k$, $w(\alpha) = w^k$, $z = z^k$, $\gamma_1 = \gamma_2 = 0$, $\gamma_3 = \gamma_3^k$, and $\varphi_{\gamma} = \varphi_{\gamma^k}$, we obtain
	\begin{equation}\label{eq:fD2}
		\big\langle\nabla f(x^{k}), w^k-x^{k} \big\rangle \leq \left(\frac{\gamma_3^k-1}{\alpha_k}\right) \|w^k-x^k\|^2.
	\end{equation}
	By \eqref{eq:TolArm}, we have $(1-\gamma_3^k)/\alpha_k \geq (1-\bar{\gamma})/\alpha_{\max}$. Thus, combining \eqref{eq:fD1} with \eqref{eq:fD2} and taking into account Lemma \ref{Le:tauminArm}, it follows that
	$$f(x^k) - f(x^{k+1}) \geq \sigma\tau_k \left(\frac{1-\bar{\gamma}}{\alpha_{\max}}\right) \|w^k-x^k\|^2 \geq \sigma\tau_{\min}\left(\frac{1-\bar{\gamma}}{\alpha_{\max}}\right) \|w^k-x^k\|^2.$$
	Hence, performing the sum of the above inequality for $k= 0, 1,\ldots, N-1$, we have
	$$
	\sum_{k= 0}^{N-1} \|w^k - x^k\|^2 \leq \frac{\alpha_{\max}\left[f(x^0) - f(x^N)\right]}{\sigma\tau_{\min}(1- \bar{\gamma})}\leq \frac{\alpha_{\max}\left[f(x^0) - f^*\right]}{\sigma\tau_{\min}(1- \bar{\gamma})}, 
	$$
	which implies the desired inequality.
\end{proof}
\subsubsection{Iteration-complexity bound under convexity} \label{SubSec:IntCompA}
In this section, we present an iteration-complexity bound for the sequence $\left(f(x^k)\right)_{k\in\mathbb{N}}$ when  $f$ is convex.
\begin{theorem}
	Let $f$ be a convex function on $C$. Then, for every $N \in \mathbb{N}$, there holds
	$$
	\min \left\{f(x^k) - f^* :~k = 0, 1 \ldots, N-1\right\} \leq \frac{\|x^0 - x^*\| + \xi\left[f(x^0)-f^*\right]}{2 \alpha_{\min} \tau_{\min}}\frac{1}{N}.
	$$
\end{theorem}

\begin{proof}
	Using the first inequality in \eqref{eq:TolArm} and Lemma~\ref{Le:tauminArm}, we have $2 \alpha_{\min} \tau_{\min} \leq 2 \alpha_k \tau_k$, for all $k\in {\mathbb N}$. From the convexity of $f$, we have $\big\langle \nabla f(x^k), x^*-x^k \big\rangle \leq f^* - f(x^k)$, for all $k\in {\mathbb N}$. Thus, applying Lemma \ref{Le:xkArm} with $x=x^*$, after some algebraic manipulations, we conclude 
	$$
	2 \alpha_{\min}\tau_{\min} \left[f(x^k)-f^*\right] \leq \|x^k-x^*\|^2-\|x^{k+1}-x^*\|^2 + \xi \left[f(x^k) - f(x^{k+1}) \right] \quad k = 0, 1, \ldots.
	$$
	Hence, performing the sum of the above inequality for $k = 0,1,\ldots, N-1$, we obtain 
	$$
	2 \alpha_{\min}\tau_{\min}\sum_{k=0}^{N-1} \left[f(x^k)-f^*\right] \leq \|x^0-x^*\|^2-\|x^N-x^*\|^2 + \xi\left[f(x^0)-f(x^N)\right].
	$$
	Therefore, $2\alpha_{\min}\tau_{\min} N \min\{f(x^k) - f^* : k = 0, 1 \ldots, N-1\} \leq \|x^0 - x^*\| + \xi\left[f(x^0)-f(x^N)\right]$, which implies the desired inequality.
\end{proof}

\section{Numerical experiments} \label{Sec:NumExp}

In this section, we summing up the results of our preliminary numerical experiments in order to verify the practical behavior  of the proposed algorithms. 
In particular, we will illustrate  the potential advantages of considering inexact projections instead of exact ones  in a problem of least squares over the spectrahedron. 
The codes are written in Matlab and are freely available at \url{https://orizon.ime.ufg.br}. All experiments were run on a  macOS 10.15.7 with 3.7GHz Intel Core i5 processor and 8GB of RAM.

Let $\mathbb{S}^n$ be the space of $n\times n$ symmetric real matrices and $\mathbb{S}^n_+$ be the cone of positive semidefinite matrices in $\mathbb{S}^n$. Given $A$ and $B$ two $n\times m$ matrices, with $m\geq n$, we consider the following problem:
\begin{equation} \label{lsspec}
 \begin{array}{cl}
\ds\min_{X\in \mathbb{S}^n}   &  f(X):=\ds\frac{1}{2}\|AX-B\|^2_F         \\
\mbox{s.t.} & \tr(X)=1, \\
                   & X \in \mathbb{S}^n_+, \\
\end{array}
\end{equation}
where $X$ is the $n\times n$ matrix that we seek to find. Here, $\|\cdot\|_F$ denotes the Frobenius matrix norm $\|A\|_F=\sqrt{\langle A,A \rangle}$, where the inner product is given by $\langle A,B \rangle = \tr(A^TB)$.
Problem \eqref{lsspec} and its variants appear in applications in different areas such as statistics, physics and  economics \cite{ESCALANTE199873,doi:10.1137/0902021,douglasprojected, WOODGATE1996171}, and were considered, for example, in the numerical tests of \cite{BirgMartRay2000,gonccalves2020inexact,Guanghui2013,LanZhou2016}.

Following we briefly discuss how to compute projections onto the feasible region of \eqref{lsspec}.
Define $C=\{X\in\R^{n\times n}\mid \tr(X)=1, \; X \in \mathbb{S}^n_+\}$.
Since $\mathbb{S}^n_+$ is a convex  and  closed set, $C$ is convex and compact. Formally the problem of projection a given vector $V\in\R^{n\times n}$ onto $C$ is stated as 
\begin{equation}\label{prob:pj}
 \begin{array}{cl}
\ds\min_{W\in \mathbb{S}^n}  &  \ds\frac{1}{2} \|W-V\|_F^2         \\
\mbox{s.t.} & W\in C. \\
\end{array}
\end{equation}
Since $\|W-V\|_F^2=\|W-V_S\|_F^2+\|V_A\|^2_F$ for any $W\in\mathbb{S}^n$, where $V_S$ and $V_A$ denote the symmetric and the antisymmetric part of $V$, respectively, it can be assumed, without loss of generality, that $V\in\mathbb{S}^n$. 
Let $W^*$ be the unique solution of \eqref{prob:pj}.
Given the eigen-decomposition $V = QDQ^T$, it is well known that $W^* = QP_{\Delta_n}(D)Q^T$, where $P_{\Delta_n}(D)$ denotes the diagonal matrix obtained by projecting the diagonal elements of $D$ onto the $n$-dimensional simplex $\Delta_n=\{x\in\R^n\mid x_1+\ldots+x_n=1; \; x\geq 0\}$, see, for example, \cite{douglasprojected}.
This means that  computing $W^*$ requires {\it a priori} the full eigen-decomposition of $V$, which can be computationally prohibitive for  high-dimensional problems. 
This drawback will appear clearly in the results reported in section \ref{sec:comparasion}. 
Inexact projections can be obtained by adding the constraint $\rank(W)\leq p$, for a given $1\leq p < n$, to Problem \eqref{prob:pj}.
Denoting by $W_p$ the solution of this latter problem, we have $W_p=\sum_{i=1}^p \lambda_iq_iq_i^T$, where the scalars $(\lambda_1,\ldots,\lambda_p)\in\R^p$ are obtained by projecting the $p$ largest eigenvalues of $V$ onto the $p$-dimensional simplex $\Delta_p$ and $q_i\in\R^n$ are the corresponding unit eigenvectors for all $i=1,\ldots,p$, see \cite{allen2017linear}. 
Therefore, inexact projections can be computed by means of an incomplete eigen-decomposition of $V$, resulting in computational savings.
Note that if $\rank(W^*)\leq p$, then $W_p$  coincides with $W^*$. As far as we know, this approach was first proposed in \cite{allen2017linear} and was also used in \cite{gonccalves2020inexact,gonccalves2019levenberg}.

We implemented the inexact projection scheme discussed above by choosing parameter $p$ in an adaptive way and introducing a suitable error criterion.
Let us formally describe the adopted scheme to find an inexact solution  $\tilde{W}$ to Problem~\eqref{prob:pj} relative to $U \in C$ with  error tolerance mapping $\varphi_{\gamma}$ and real numbers $\gamma_1$, $\gamma_2$ and $\gamma_3$ satisfying the suitable conditions \eqref{eq:Tolerance} or \eqref{eq:TolArm}, depending on the main algorithm. 

\begin{algorithm}[H]
	\caption{Procedure to compute $\tilde{W}\in  {\cal P}_C\left(\varphi_{\gamma}, U, V \right)$}
	\label{Alg:CondG}
	\begin{algorithmic}		
		\Input{Let $1\leq p <n$ be given.}
		\State \textbf{Step 1:} Compute $(\lambda_i^V,q_i)_{i=1}^p$ (with $\|q_i\|=1$, $i=1,\ldots,p$) the $p$ largest eigenpairs of $V$, then set 
		$$W_p:=\sum_{i=1}^p \lambda_iq_iq_i^T,$$
		where $(\lambda_1,\ldots,\lambda_p)\in\R^p$ are obtained by projecting $(\lambda_1^V,\ldots,\lambda_p^V)\in\R^p$ onto the $p$-dimensional simplex $\Delta_p$.
		\State \textbf{Step 2:} Compute $Y_{p} := \ds\arg\min_{Y \in  C} \, \langle W_p  - V, Y - W_p  \rangle$.
		\State \textbf{Step 3:} If $\langle W_p  - V, Y_{p}  - W_p  \rangle \geq -\varphi_{\gamma}(U,V,W_p )$, then set $\tilde{W}:=W_p$ and {\bf return} to the main algorithm.
		\State \textbf{Step 4:} Set $p\leftarrow p+1$ and go to {\bf Step~1}.
		\Output{$\tilde{W}:=W_p$.}
	\end{algorithmic}
\end{algorithm}

Some comments regarding Algorithm~\ref{Alg:CondG} are in order. First, in Step~1 we compute the rank-$p$ projection of $V$ onto $C$.  The computational cost os this step is dominated by the cost of computing the $p$ leading eigenpairs of $V$, since the projection of $(\lambda_1^V,\ldots,\lambda_p^V)\in\R^p$ onto the $p$-dimensional simplex $\Delta_p$ can be easily done in  $O(p \log p)$ time, see, for example, \cite{allen2017linear}.
Second, the subproblem in Step~2 is solved by computing the largest eigenpair of  $V-W_p$. Indeed, $Y_{p}=qq^T$, where $q\in\R^n$ is the unit eigenvector corresponding to the largest eigenvalue of $V-W_p$, see, also, \cite{allen2017linear}. In our implementations, we used the Matlab function {\it eigs} to compute eigenvalues/eigenvectors \cite{eigs1,eigs2}.
Third, in Step~3 if the stopping criterion $\langle W_p  - V, Y_{p}  - W_p  \rangle \geq -\varphi_{\gamma}(U,V,W_p )$ is satisfied, then from Definition \ref{def:InexactProj}, we conclude that $\tilde{W}=W_p\in  {\cal P}_C\left(\varphi_{\gamma}, U, V \right)$, i.e., the output is a feasible inexact projection of $V\in\mathbb{S}^n$ relative to $U \in C$. Otherwise, we increase parameter $p$ and proceed to calculate a more accurate eigen-decomposition of $V$.
Fourth, in the first iteration of the main algorithm, we set $p = 1$ as the input for Algorithm~\ref{Alg:CondG}. For the subsequent iterations, we used, in principle, the {\it success value} for $p$ from the previous outer iteration. 
Without attempting to go into details, seeking computational savings, in some iterations, we consider decreasing the input $p$ with respect to the previous successful one.

Concerning the stopping criterion of the main algorithms, all runs were stopped at an iterate $X^k$ declaring convergence if
$$\frac{\|X^{\ell}-X^{\ell-1}\|_F}{\|X^{\ell-1}\|_F}\leq 10^{-4},$$ 
for $\ell=k$ and   $\ell=k-1$. This means that we stopped the execution of the main algorithms when the above convergence metric is satisfied for two consecutive iterations.

We used a similar strategy as in \cite{gonccalves2020inexact} to generate the test instances of \eqref{lsspec}. Given the dimensions $n$ and $m$, with $m\geq n$, we randomly generate $A$ (a sparse matrix with density $10^{-4}$) with elements between $(-1,1)$. Also, for a give parameter $\omega>1$, we define $\bar{X}:=\sum_{i=1}^{\omega} g_ig_i^T$, where $g_i\in\R^n$ is a random vector with only two nonnull components with the following structure $g_i=(\cdots \cos(\theta) \cdots  \sin(\theta)  \cdots)^T\in\R^n$, and then set $B=A\bar{X}$. Since $\bar{X}\notin C$, this procedure generally results in nonzero residue problems.


\subsection{Influence of the forcing parameter $\gamma$}

We start the numerical experiments by checking the influence of the forcing parameter  $\gamma^k = (\gamma_1^k, \gamma_2^k, \gamma_3^k)$ in  Algorithm~\ref{Alg:INP}.
We implemented Algorithm~\ref{Alg:INP} with: $(a_k)_{k\in\mathbb{N}}$ and $(b_k)_{k\in\mathbb{N}}$ given as in Remark~\ref{re:csab}(ii) with $\bar{b}=100$, and using
\begin{equation}\label{phi}
\varphi_{\gamma^k}(U,V,W) = \gamma_1^k \|V-W\|_F^2 + \gamma_2^k \|W-V\|_F^2 + \gamma_3^k \|W-U\|_F^2, \quad \forall k=1,2,\ldots,
\end{equation}
as the error tolerance function, see Definition~\ref{def:InexactProj}. We also set $\bar{\gamma_2} = 0.49995$.
Concerning parameter $\bar{\gamma}$, we considered different values for it, as we will explain below.
Given a particular $\bar{\gamma}<1/2$, we set
\begin{equation}\label{alphachoice}
\alpha =  0.9999 \cdot \frac{1-2 \bar{\gamma}}{L},
\end{equation}
where the Lipschitz constant $L$, with respect to problem \eqref{lsspec}, is given by $L=\|A^TA\|_F$.
The choice \eqref{alphachoice} for the fixed step size  $\alpha$ trivially satisfies \eqref{eq:alpha}.
Since  parameter $\gamma_3^k$ and the step size $\alpha$ are closely related (see \eqref{alphachoice} and recall that $0 \leq \gamma_3^k \leq \bar{\gamma} < 1/2$), we first investigate the behavior of Algorithm~\ref{Alg:INP} by varying only the strategy for $\gamma_3^k$. We set
\begin{equation}\label{gamma12}
\gamma_2^k=\min\left( \frac{1}{2}\frac{a_k}{\|\nabla f(X^{k})\|_F^2} , \bar{\gamma_2}\right), \quad \gamma_1^k=\frac{a_k}{\|\nabla f(X^{k})\|_F^2} -\gamma_2^k,  \quad \mbox{and} \quad \gamma_3^k = \bar{\gamma}, \quad \forall k=1,2,\ldots,
\end{equation}
and considered some different values for $\bar{\gamma}$. Note that the forcing parameter $\gamma^k$ given by \eqref{gamma12} satisfies  \eqref{eq:Tolerance}.
We used an instance of problem \eqref{lsspec} with $n=2000$, $m=4000$, and $\omega = 10$. The results for the starting point $X^0= (1/n)I$ and different choices for $\bar{\gamma}$ are in Table~\ref{tab:gamma3}. In the table, ``$f(X^*)$''  is the function value at the final iterate, ``it'' is the number of outer iterations, ``time(s)'' is the run time in seconds, and ``$\alpha$'' is the corresponding fixed step size given by \eqref{alphachoice}.

\begin{table}[H]
\centering
{\footnotesize
\begin{tabular}{|c|cccc|} \hline
\cellcolor[gray]{.85}  $\gamma_3^k= \bar{\gamma}$&\cellcolor[gray]{.85} $f(X^*)$ & \cellcolor[gray]{.85} it &\cellcolor[gray]{.85} time(s)&\cellcolor[gray]{.85} $\alpha$\\ \hline
 $ 0.0$ & 0.4899 & 107 & 28.9  & 0.0698\\
 $ 0.1$ & 0.4899 & 129 & 36.6 & 0.0558\\
 $ 0.2$ & 0.4899 & 162 & 43.5  & 0.0419 \\
 $ 0.3$ & 0.4899 & 223 & 59.6  & 0.0279 \\
 $ 0.4$ & 0.4899 & 375 & 101.0 & 0.0140 \\ \hline
\end{tabular}
}
\caption{Influence of parameter $\bar\gamma_3$ in the performance of Algorithm~\ref{Alg:INP} with the forcing parameter $\gamma^k$ given as in \eqref{gamma12} for an instance of problem \eqref{lsspec} with $n=2000$, $m=4000$, and $\omega = 10$.}
\label{tab:gamma3}
\end{table}

As can be seen in Table~\ref{tab:gamma3}, Algorithm~\ref{Alg:INP} performs better for lower values of $\bar\gamma_3$.
This is undoubtedly due to the fact that the fixed step size $\alpha$ is inversely proportional to $\bar\gamma_3$, see \eqref{alphachoice} and the last column of the table.
This result suggests that for the gradient method with constant step size, the best choice is to take $\gamma_3^k = 0$ for all $k$, leaving the inexactness of the projections to be controlled only by the terms of $\gamma_1^k$ and $\gamma_2^k$ in \eqref{phi}.
From an algorithmic point of view, the term corresponding to $\gamma_3^k$ in \eqref{phi} involves the last two generated iterates  and is often used in inexactness measures.
Therefore, for projection algorithms that use a constant step size, at least under restrictions as in \eqref{eq:alpha}, the theory developed here presents practical alternatives for the formulation of such measures.

Taking $\bar{\gamma}=0$, we consider different combinations of $\gamma_1^k$ and $\gamma_2^k$ such that $\gamma_1^k+\gamma_2^k=a_k/\|\nabla f(X^{k})\|_F^2$.
Our experiments showed that Algorithm~\ref{Alg:INP} presented no significant performance difference with these combinations.
Therefore, in the experiments reported in section~\ref{sec:comparasion}, we set for Algorithm~\ref{Alg:INP} the forcing parameter $\gamma^k$ as in \eqref{gamma12} with $\bar{\gamma}=0$.

\subsection{Comparison with exact projection approaches} \label{sec:comparasion}

In the present section, we compare the performance of Algorithms~\ref{Alg:INP} and \ref{Alg:Armijo} with their exact counterparts. Algorithm~\ref{Alg:Armijo} was implemented using the error tolerance function \eqref{phi} with 
$$\gamma^k =(0,0,0.49995), \quad \forall k=1,2,\ldots,$$
 and
\begin{equation}\label{spg}
\alpha_k:=\left\{\begin{array}{ll}
\ds\min\left(\alpha_{\max},\max \left(\alpha_{\min},\langle S^k,S^k\rangle/\langle S^k,Y^k\rangle\right)\right), & \mbox{if} \; \langle S^k,Y^k\rangle > 0\\
\alpha_{\max}, & \mbox{otherwise},
\end{array}\right.
\end{equation}
where $S^k:=X^k - X^{k-1}$, $Y^k:=\nabla f(X^k) - \nabla f(X^{k-1})$, $\alpha_{\min}=10^{-10}$, and $\alpha_{\max}=10^{10}$.
We observe that \eqref{spg} corresponds to the spectral choice for $\alpha_k$, see \cite{BirgMartRay2000,spgsiam}.
In the exact versions, the projections are calculated exactly, that is, involving full eigen-decompositions.

We considered some instances of problem \eqref{lsspec} with different parameters $n$, $m$ and $\omega$ and using three starting points given by $X^0(\beta)=(1-\beta)(1/n)I+\beta e_1e_1^T$, where $e_1\in\R^n$ is the first canonical vector and $\beta\in\{0.00,0.50,0.99\}$. The results in Table~\ref{tab:results} shows that, in relation to the CPU time, the inexact algorithms were notably more efficient (mainly in the larger instances) than the corresponding exact versions. In general, moderate values for the rank parameter $p$ in Algorithm~\ref{Alg:CondG} (typically, less than 10) were sufficient to compute the inexact projections, allowing significant computational savings with respect to the exact approaches.
Finally, we observe that Algorithm~\ref{Alg:Armijo} was much more efficient  than Algorithm~\ref{Alg:INP} on the chosen set of test problems.
This was already expected, due to the simplicity of the objective function of \eqref{lsspec}. Note that Algorithm~\ref{Alg:INP} does not require evaluations of the objective function (only gradient evaluations). Therefore, we hope that Algorithm~\ref{Alg:INP} can be competitive in problems where the objective function is expensive to be computationally evaluated.

\begin{table}[H]
\centering 
{\tiny
\begin{tabular}{|cc|c|c|ccc|ccc|ccc|ccc|}\hhline{*4{~}|*6{-}*6{-}|} 
  \multicolumn{4}{c|}{} & \multicolumn{6}{c|}{\cellcolor[gray]{.85} Algorithm~\ref{Alg:INP}} & \multicolumn{6}{c|}{\cellcolor[gray]{.85}  Algorithm~\ref{Alg:Armijo}}\\  \hhline{*4{~}|*6{-}*6{-}|} 
  \multicolumn{4}{c|}{} & \multicolumn{3}{c|}{\cellcolor[gray]{.85} Inexact} & \multicolumn{3}{c|}{\cellcolor[gray]{.85}  Exact} & \multicolumn{3}{c|}{\cellcolor[gray]{.85}  Inexact } & \multicolumn{3}{c|}{\cellcolor[gray]{.85}  Exact} \\ \hline \rowcolor[gray]{.85}
  $n$ & $m$ & $\omega$ & $\beta$ & $f(X^*)$ & it & time(s) & $f(X^*)$ & it & time(s) & $f(X^*)$ & it & time(s) & $f(X^*)$ & it & time(s) \\ \hline
          &          &      & 0.00&   0.4899 &  107 & 30.0 & 0.4899 &  108 & 78.3 & 0.4899 &    8 & 3.6 & 0.4899 &    8 & 6.8 \\ 
          &          & 10 & 0.50 &   0.4899 &  108 & 36.7 &  0.4899 &  108 & 78.4 & 0.4899 &    9 & 4.0 & 0.4899 &    9 & 7.7 \\ 
          &          &      &  0.99 &   0.4899 &  110 & 29.2 &  0.4899 &  110 & 78.8 & 0.4899 &    9 & 4.0 & 0.4899 &    9 & 7.6 \\   \hhline{*2{~}|*8{-}*6{-}|} 
          &          &      & 0.00 &   0.7887 &  141 & 47.6 &  0.7887 &  141 & 101.7 & 0.7887 &    9 & 4.2 & 0.7887 &    9 & 7.5 \\ 
  &  & 20 & 0.50 &   0.7887 &  139 & 46.8 &  0.7887 &  139 & 100.4 & 0.7887 &    9 & 4.2 & 0.7887 &    9 & 7.6 \\ 
 \multirowcell{-6}{2000}          & \multirowcell{-6}{4000}           &      &  0.99 &   0.7887 &  136 & 36.1 &  0.7887 &  138 & 96.6 & 0.7887 &   10 & 4.8 & 0.7887 &   10 & 8.3 \\   \hline
          &          &      & 0.00&   0.2139 &  189 & 124.3 &  0.2139 &  189 & 543.1 & 0.2139 &    8 & 9.6 & 0.2139 &    8 & 27.0 \\ 
  &  & 10 & 0.50 &   0.2139 &  189 & 133.2 &  0.2139 &  189 & 541.5 & 0.2139 &    8 & 9.6 & 0.2139 &    9 & 30.0 \\ 
          &          &      &  0.99 &   0.2139 &  190 & 121.8 &  0.2139 &  190 & 537.1 & 0.2139 &   10 & 11.5 & 0.2139 &    9 & 30.2 \\  \hhline{*2{~}|*8{-}*6{-}|} 
          &          &      & 0.00 &   0.9837 &  173 & 113.1 &  0.9837 &  172 & 491.6 & 0.9837 &   11 & 13.0 & 0.9837 &   12 & 39.4 \\ 
  &  & 20 & 0.50 &   0.9837 &  170 & 109.9 &  0.9837 &  170 & 485.2 & 0.9837 &   11 & 12.9 & 0.9837 &   10 & 33.2 \\ 
 \multirowcell{-6}{3000}          & \multirowcell{-6}{6000}           &      &  0.99 &   0.9837 &  161 & 102.8 &  0.9837 &  164 & 466.3 & 0.9837 &   12 & 14.1 & 0.9837 &   10 & 33.3 \\   \hline
          &          &      & 0.00 &   1.0046 &  166 & 194.2 &  1.0046 &  165 & 1092.5 & 1.0046 &   10 & 22.8 & 1.0046 &   11 & 83.7 \\
  &  & 10 & 0.50 &   1.0046 &  165 & 191.6 &  1.0046 &  165 & 1088.8 & 1.0046 &   12 & 26.4 & 1.0046 &   11 & 83.7 \\ 
          &          &      &  0.99 &   1.0046 &  169 & 193.4 &  1.0046 &  169 & 1113.0 & 1.0046 &   11 & 25.0 & 1.0046 &   12 & 91.2 \\   \hhline{*2{~}|*8{-}*6{-}|} 
          &          &      & 0.00 &   3.0753 &   90 & 126.2 &  3.0753 &   89 & 585.9 & 3.0753 &    8 & 16.9 & 3.0753 &    8 & 63.6 \\ 
  &  & 20 & 0.50 &   3.0753 &   89 & 122.8 &  3.0753 &   89 & 584.9 & 3.0753 &    9 & 18.5 & 3.0753 &    8 & 62.0 \\ 
 \multirowcell{-6}{4000}          & \multirowcell{-6}{8000}           &      &  0.99 &   3.0753 &   87 & 99.4 &  3.0753 &   86 & 566.0 & 3.0753 &    9 & 18.8 & 3.0753 &    9 & 67.3 \\   \hline
          &          &        & 0.00 &  0.7182 &  243 & 599.6 &  0.7182 &  243 & 3212.4 & 0.7181 &   10 & 39.7 & 0.7181 &   10 & 150.8 \\  
  &  & 10 & 0.50 &  0.7182 &  244 & 594.4 &  0.7182 &  244 & 3206.3 & 0.7181 &    9 & 36.3 & 0.7181 &   10 & 151.4 \\ 
          &          &        &  0.99 &  0.7182 &  246 & 571.2 &  0.7182 &  246 & 3227.1 & 0.7181 &   10 & 38.2 & 0.7181 &   10 & 150.0 \\   \hhline{*2{~}|*8{-}*6{-}|} 
          &          &        & 0.00 &   2.7721 &  178 & 436.3 &  2.7721 &  178 & 2339.8 & 2.7721 &    8 & 29.8 & 2.7721 &    8 & 122.8 \\  
  &  & 20 & 0.50 &   2.7721 &  177 & 436.2 &  2.7721 &  177 & 2325.5 & 2.7721 &    8 & 30.5 & 2.7721 &    7 & 108.5 \\ 
 \multirowcell{-6}{5000}          & \multirowcell{-6}{10000}           &        &  0.99 &   2.7721 &  172 & 395.9 &  2.7721 &  172 & 2253.6 & 2.7721 &    8 & 29.6 & 2.7721 &    8 & 122.3 \\  \hline
\end{tabular}
}
\caption{Performance of the inexact and exact versions of Algorithms~\ref{Alg:INP} and \ref{Alg:Armijo}  in some instances of problem  \eqref{lsspec}.}
\label{tab:results}
\end{table}

\section{Conclusions} \label{Sec:Conclusions}
In this paper, we proposed a new inexact version of the classical gradient projection method (GPM) denoted by Gradient-InexP method (GInexPM) for solving constrained convex optimization problems. As a way to compute an inexact projection the GInexPM uses a relative error tolerance. Two different strategies for choosing the step size were employed in the analyses of the method.	The convergence analysis was carried out without any compactness assumption. In addition, we provided some iteration-complexity results related to GInexPM.
Numerical results were reported illustrating  potential advantages of considering inexact projections instead of exact ones. We expect that this paper will contribute to the development of research in this field of inexact projections, mainly to solve large-scale problems.


\end{document}